\documentclass{amsart}
\usepackage{pgf,pgfarrows,pgfnodes,pgfautomata,pgfheaps,pgfshade,hyperref, amssymb,enumerate,amsmath}
\usepackage[all]{xy}
\usepackage[capitalize]{cleveref}
\usepackage{mathtools}
\usepackage[shortlabels]{enumitem}
\usepackage{stmaryrd}

\newtheorem{theorem}{Theorem}[section]
\newtheorem*{theorem*}{theorem}
\newtheorem{proposition}[theorem]{Proposition}
\newtheorem{lemma}[theorem]{Lemma}

\newtheorem{Rem}[theorem]{Remark}

\newtheorem{definition}[theorem]{Definition}

\usepackage{euscript,epsfig}
\usepackage{tikz}
\usetikzlibrary{graphs}
\usetikzlibrary[graphs]
\usetikzlibrary{arrows}
\usepackage{tikz-cd}
\usetikzlibrary{decorations.markings}
\usepackage[colorinlistoftodos]{todonotes}

\def\R{{\mathbb R}}
\def\Z{{\mathbb Z}}

\def\N{{\mathbb N}}

\def\QQ{{\mathbb Q}}
\def\cD{{\mathcal D}}

\def\cB{{\mathcal B}}
\def\cC{{\mathcal C}}
\def\cA{{\mathcal A}}
\def\cL{{\mathcal L}}

\def\cM{{\mathcal M}}

\def\cS{{\mathcal S}}

\newcommand{\bt}{\boldsymbol\tau}
\newcommand{\bs}{\boldsymbol\tau}
\DeclareMathOperator{\Id}{Id}
\DeclareMathOperator{\diag}{diag}
\DeclareMathOperator{\Aff}{Aff}

\DeclareMathOperator{\Inf}{Inf}

\author{Paulina Cecchi Bernales}%
\address{Centro de Modelamiento Matem\'atico, Universidad de Chile \&  UMI-CNRS 2807,  Beauchef 851, Santiago, Chile.}
\email{pcecchi@dim.uchile.cl}

\author{Sebasti\'an Donoso}
\address{Departamento de Ingenier\'{\i}a Matem\'atica and Centro de Modelamiento Matem\'atico, Universidad de Chile \& UMI-CNRS 2807, Beauchef 851, Santiago, Chile.}
\email{sdonoso@dim.uchile.cl}

\thanks{ This work was supported by Centro de Modelamiento Matem\'atico (CMM), ACE210010 and FB210005, BASAL funds for centers of excellence from ANID-Chile. The first author was supported by ANID/Fondecyt/3210746 and the second author by ANID/Fondecyt/1200897.}% and Grant Basal-ANID AFB170001.}

\begin{document}

\title{Strong orbit equivalence and superlinear complexity}



\date{\today}

\begin{abstract}
We show that within any strong orbit equivalent class, there exist minimal subshifts with arbitrarily low superlinear complexity. This is done by proving that for any simple dimension group with unit $(G,G^+,u)$ and any sequence of positive numbers $(p_n)_{n\in \N}$ such that $\lim n/p_n=0$, there exist a minimal subshift whose dimension group is order isomorphic to $(G,G^+,u)$ and whose complexity function grows slower than $p_n$. As a consequence, we get that any Choquet simplex can be realized as the set of invariant measures of a minimal Toeplitz subshift whose complexity grows slower than $p_n$.
\end{abstract}

\maketitle

\section{Introduction}\label{sec:intro}

Two minimal topological dynamical systems are orbit equivalent if there is a homeomorphism between the phase spaces sending orbits to orbits. They are strongly orbit equivalent if they are orbit equivalent and the {\em cocycle} function is discontinuous at most in one point (see \cref{subsec:oe} for precise definitions). When the strong orbit equivalence property was introduced in topological dynamics, it was not clear its relationship with entropy, in particular, it was asked if entropy was preserved under strong orbit equivalence. This was dramatically disproved by Boyle and Handelman \cite{boylehandelman}, who showed that any Cantor minimal system (a system where the phase space is a Cantor space) is strongly orbit equivalent to a zero entropy Cantor system. Later, in a series of papers, Sugisaki \cite{sugisaki:1998,sugisaki:2003,sugisaki:2007} generalized this result, showing among other things that for any $c\in [0,+\infty]$ any Cantor minimal system is strongly orbit equivalent to a system whose topological entropy equals $c$. All these results make clear that entropy poses no restriction whatsoever to being strongly orbit equivalent. So one could ask, what are the natural restrictions that arise from being strongly orbit equivalent. A fundamental object that appears in this context is the dimension group of a dynamical system, an algebraic concept imported from the study of $C^{\ast}$-algebras \cite{vershik}, \cite{HPS}. In \cite{GPS} Giordano, Putnam and Skau  showed that two Cantor minimal systems are strongly orbit equivalent if and only if they have homeomorphic dimension groups with unit (we refer to \cref{sec:basicdef} for definitions). In particular, entropy does not impose any restriction on the dimension group of a Cantor minimal system.

One could consider a finer notion than entropy and study how it restricts the dimension group of a Cantor minimal system. In the case of a subshift, we focus our attention on the {\em complexity function}, defined as the function $p_X\colon \N\to \N$ that counts  the  number  of  non-empty  cylinders  of  length $n$. The topological entropy of the subshift is nothing but the exponential growth rate of $p_X$.
The complexity function has been widely studied in the past decades, a survey is provided in \cite{fe99}. Note that among zero entropy systems, one can still observe a wide variety of behaviors of $p_X$ for different subshifts, ranging from bounded (when $X$ is a finite system), sublinear (e.g. primitive substitutions, see for example \cite{Que}) till superpolynomial subexponential functions. First, notice that the complexity function does impose some restrictions on the dimension group of a minimal subshift. A trivial example is when $p_X$ is bounded, since that implies that $X$ is finite \cite{HM}, \cite{HM2}, and its dimension group is isomorphic to $(\Z,\Z^+,1)$\footnote{Here, the minimality assumption implies that the system is a periodic orbit}. A  non-trivial restriction occurs when the minimal subshift has non-superlinear complexity,{\it i.e.}, $\liminf p_X(n)/n<+\infty$. In \cite{DDMP20}, such a system was proved to be of {\em finite rank} (they admit a Bratteli-Vershik representation with a bounded number of vertices at each level) and therefore its dimension group is severely constrained: it is an abelian group whose rational rank is finite.

Since having non-superlinear complexity constrains the dimension group of a minimal subshift we can ask: do other growths of the  complexity function constrain the dimension group of a minimal subshift? For instance, is there anything particular about the dimension group of a minimal subshift with subquadratic complexity? Our main result is that except for non-superlinearity, there is no restriction on the complexity of a minimal subshift with a given dimension group. 

\begin{theorem}\label{theo:main1}
Let $(G,G^+,u)$ be a simple dimension group with unit, and let $p_n$ be a sequence of positive real numbers such that $ \lim n/p_n=0$. Then there exists a minimal subshift $(Y,S)$ such that 
\begin{enumerate}
    \item $K^0(Y,S)$ and $(G,G^+,u)$ are isomorphic as ordered groups with unit.
    \item The complexity of $(Y,S)$ satisfies $\lim p_{Y}(n)/p_n=0$. 
\end{enumerate} 
%Moreover if $G$ is divisible, $(Y,S)$ can be chosen to be Toeplitz.
\end{theorem}
In terms of strong orbit equivalence \cref{theo:main1} can be restated as:
\begin{theorem}\label{theo:main2}
Let $(X,T)$ be a Cantor minimal system and let $p_n$ be a sequence of positive real numbers such that $ \lim n/p_n= 0$. Then there exists a minimal subshift $(Y,S)$ such that 
\begin{enumerate}
    \item (Y,S) and (X,T) are strong orbit equivalent.
    \item The complexity of $(Y,S)$ satisfies $\lim p_{Y}(n)/p_n=0$. 
\end{enumerate}
\end{theorem}

With some extra conditions on the dimension group,  we can build a Toeplitz subshift having low complexity.  

\begin{theorem} \label{thm:divisible_toeplitz}
Let $(G,G^+,u)$ be a simple dimension group with unit, and let $p_n$ be a sequence of positive real numbers such that $ \lim n/p_n=0$. Assume that $G$ is a divisible group. Then there exists a minimal Toeplitz subshift $(Y,S)$ such that 
\begin{enumerate}
    \item $K^0(Y,S)$ and $(G,G^+,u)$ are isomorphic as ordered groups with unit.
    \item The complexity of $(Y,S)$ satisfies $\lim p_{Y}(n)/p_n=0$. 
\end{enumerate} 
\end{theorem}

\subsection*{Invariant measures}
The Choquet simplex of invariant measures of a Cantor minimal system is affine homeomorphic to the set of {\it traces} of the dimension group of the system \cite{HPS} (see \cref{subsec:dg} for precise definitions). Thanks to the characterization of strong orbit equivalence through the dimension group \cite{GPS}, this implies that the simplex of invariant measures, modulo affine homeomorphism, is preserved under strong orbit equivalence. Many works have been developed to study the properties of the set of invariant measures of a Cantor minimal system and the connections between those properties and some dynamical properties of the system \cite{Wi, Do91, bezu1, bezu2}. In \cite{Do91} Downarowicz showed that a Toeplitz subshift may have the affine structure of an arbitrary Choquet simplex. This realization problem has been later extended to more general group actions on the Cantor space \cite{CP14}, \cite{CC19}.

Connections between the complexity of a subshift and its set of invariant measures have been explored as well. By a theorem of Boshernitzan \cite{Bosh}, a subshift having a non-superlinear complexity has finitely many ergodic measures. This can also be seen as a consequence of the fact that such systems are of finite rank \cite{DDMP20}. Some recent works have been devoted to getting better upper bounds to the number of ergodic and generic measures \cite{cyrkra1,DF:2017,DOP:2020}. Recently Cyr and Kra \cite{cyrkra2} showed that for any superlinear sequence $(p_n)_{n\in \N}$ there exists a minimal subshift $(X,S)$ such that $\liminf p_X(n)/p_n=0$ having uncountably many ergodic measures (see \cref{theo:ck2}).  As a consequence of \cref{thm:divisible_toeplitz} we generalize this result obtaining that any Choquet simplex can be realized as the set of invariant measures of a Toeplitz subshift with arbitrarily low, superlinear complexity. This can also be seen as a generalization of the result of Downarowicz \cite{Do91} in the sense that Toeplitz subshifts can not only realize any Choquet simplex as set of invariant measures, but they can do it having an arbitrarily low superlinear complexity.

\begin{theorem}\label{theo:main3} 
Let $K$ be a Choquet simplex. Let $(p_n)_{n\in \N}$ be a sequence of positive real numbers such that $ \lim n/p_n=0$. Then, there exists a Toeplitz subshift $(X,S)$ whose complexity $p_X$ satisfies $\lim p_{X}(n)/p_n=0$
and whose set of invariant measures $\cM(X,S)$ is affine homeomorphic to $K$. 
\end{theorem}

\subsection{Methods and organization} Our main tools are the use of Bratteli-Vershik models for Cantor minimal systems, together with their associated representation as $\cS$-adic subshifts, which are systems obtained by performing an infinite composition of substitutions or morphisms (see \cref{sec:brattelidef} for precise definitions). Some estimations on the complexity of minimal $\cS$-adic subshifts have been obtained in \cite{DDMP20} and we use them to bound the complexity of the subshift that will realize a given simple dimension group with unit as its dimension group.

The document is organized as follows. In \cref{sec:basicdef} we give the basic background on symbolic dynamics and more specific facts about dimension groups.  In \cref{sec:brattelidef} we introduce the tools from Bratteli-Vershik diagrams and their associated $\mathcal{S}$-adic subshifts together with complexity results in this context. \cref{sec:proofs} is devoted to proving the main results.

Throughout the article, we let $\N, \N^{\ast}$ and $\Z$ denote the set of non-negative integers, the set of positive integers, and the set of integer numbers respectively. 

\section{Preliminaries}\label{sec:basicdef}

\subsection{Cantor minimal systems.}\label{subsec:cantorms}
Through this paper, a {\it topological dynamical system} or simply a {\it dynamical system} is a pair $(X,T)$ where $X$ is a compact metric space and $T$ is a homeomorphism from $X$ to itself. When $X$ is a {\it Cantor space}, that is, $X$ is a non-empty, metrizable, perfect and totally disconnected compact space, we say that $(X,T)$ is a {\it Cantor dynamical system}. A dynamical system $(X,T)$ is {\it minimal} if it has no non-trivial closed $T$-invariant subsets. Equivalently, $(X,T)$ is minimal if for all $x\in X$, the {\it orbit} of $x$ under $T$, $O_T(x):=\{T^k(x)\colon k\in\Z\}$, is dense in $X$. If $X$ is a Cantor space and $(X,T)$ is minimal, $(X,T)$ is called a {\it Cantor minimal system}. A dynamical system $(X,T)$ is {\it expansive} if there exists $c>0$ such that for every $x\neq y$, $\sup_k d(T^k(x),T^k(y))>c$, where $d$ is a metric on $X$. %Expansiveness does not depend on the metric $d$, although the {\it expansiveness constant} $c$ does.
Two topological dynamical systems $(X_1,T_1)$ and $(X_2,T_2)$ are {\em conjugate} if there exists a homeomorphism $h\colon X_1\to X_2$ such that $h(T_1(x))=T_2(h(x))$ for all $x\in X_1$. 

Given a dynamical system $(X,T)$, let $C(X,\R)$ denote the additive group of continuous functions from $X$ to $\R$. The {\it coboundary map} $\beta\colon C(X,\R)\to C(X,\R)$ is defined by $\beta f\coloneqq f\circ T - f$.

\subsection{Subshifts and complexity}\label{subsec:subshifts}
Let $\cA$ be a finite non-empty set with cardinality $|\cA|$, which we call an {\it alphabet}.  The free monoid $\cA^{\star}$ is the set of all concatenations of symbols of $\cA$ (called {\it words}) including the empty word, which we denote $\varepsilon$. For $w \in \cA^{\ast}$, $|w|$ stands for the length of $w$, that is, the total number of letters appearing on $w$. We use the convention $|\varepsilon|=0$. The sets $\cA^\Z$ and $\cA^\N$ endowed with the product topology of the discrete topology on each copy of $\cA$ are compact metric spaces. The elements of $\cA^\Z$ and $\cA^\N$ are sometimes called {\it infinite words} on $\cA$. A {\it factor} of a finite word $w\in\cA^{\ast}$ is defined as a finite concatenation of some consecutive letters occurring in $w$. A factor of an infinite word is defined in the same way. The set of factors $\cL_x$ of $x\in \cA^\Z$ or $\cA^\N$ is called its {\it language}. If $|\cA|\geq 2$, the space $\cA^\Z$ is a Cantor space.

Let $S\colon \cA^\Z\to \cA^\Z$ denote the {\it shift} transformation, which is defined by $S((x_i)_{i \in \Z})=(x_{i+1})_{i \in \Z}$. A {\it subshift} on $\cA$ is a dynamical system given by a pair $(X,S|_X)$, where $X$ is a closed shift-invariant subset of $\cA^\Z$, endowed with the induced topology. We denote by $(X,S)$ the system $(X,S|_X)$ to avoid an overcharged notation. We use the term {\it subshift} indistinctly to name both the space $X\subseteq \cA^\Z$ and the dynamical system $(X,S)$. Given any infinite word $x$ in $\cA^\Z$ (or in $\cA^\N$), we define $(X_x,S)$ to be the {\it subshift generated by} $x$. That is,   
$$X_x=\{y\in \cA^\Z: \forall w, w \prec y \implies  w\prec x \},$$
where the symbol $\prec$ means ``is a factor of". Equivalently, $X_x$ is the closure of the orbit of $x$ under the action of the shift $S$.
If $(X,S)$ is a subshift on $\cA$, its {\it language} $\cL_X$ is defined as the set of factors of elements of $X$.

The {\it factor complexity} or simply the {\it complexity} of a subshift $(X,S)$ is the function $p_X\colon \N\to\N$ given by the number of factors of length $n$ in $\cL_X$. A subshift $(X,S)$ is said to have {\it non-superlinear factor complexity} if 
\[ \liminf \frac{p_X(n)}{n} < +\infty. \]

\subsection{Toeplitz subshifts}\label{subsec:toeplitz}
An infinite word $(x_i)_{i\in\Z}$ with symbols in the alphabet $\cA$ is said to be a {\it Toeplitz sequence} if for all $i\in\N$ there exists $p\in\N$ such that for all $k\in \Z$, $x_i=x_{i+kp}$. A subshift $(X,S)$ is called a {\it Toeplitz subshift} if $X$ is the orbit  closure under the shift of some Toeplitz sequence. Toeplitz subshifts were originally introduced in \cite{JK}. They are known to be minimal (see for example \cite{Do95}). 

\subsection{Invariant measures.}\label{subsec:im}
Given a topological dynamical system $(X,T)$, an {\it invariant measure} of $(X,T)$ is a Borel probability measure on $X$ such that for every Borel subset $A\subseteq X$, $\mu(T^{-1}(A))=\mu(A)$. The set of all invariant measures of $(X,T)$ is denoted $\cM(X,T)$. We say that a measure $\mu\in\cM(X,T)$ is {\it ergodic} if whenever $T(A)=A$ for some Borel set $A\subseteq X$, either $\mu(A)=0$ or $\mu(A)=1$. We say that $(X,T)$ is {\it uniquely ergodic} if $\cM(X,T)$ consists of one element. The set $\cM(X,T)$ is known to be non-empty and a {\it Choquet simplex}, that is, a compact convex metrizable subset $K$ of a locally convex real vector space such that for each $v\in K$ there is a unique probability measure $m$ supported on the extreme points of $K$ with $\int_{ext(K)}xdm(x)=v$. The extreme points of $\cM(X,T)$ are exactly the ergodic measures on $X$ (see for example \cite{Glasner}). Here, $\cM(X,T)$ is considered as a subspace of the dual $C(X,\R)$ endowed with the weak$^\star$ topology.

The following result, due to Downarowicz \cite{Do91} tells us that Toeplitz subshifts can realize any Choquet simplex as its set of invariant measures (see also \cite[Theorem 12]{noruegos}).
\begin{theorem}\cite[Theorem 5]{Do91}\label{theo:do91}
Let $K$ be a Choquet simplex. There exists a Toeplitz subshift $(X,S)$ such that $\cM(X,S)$ is affine homeomorphic to $K$.
\end{theorem}

\subsubsection{Invariant measures and complexity}\label{subsec:imandcomplexity}
The following condition, called the Boshernitzan condition, states that the factor complexity of a minimal subshift restricts the possible affine structure of its set of invariant measures. More precisely, it states that minimal subshifts with non-superlinear factor complexity have a finite number of ergodic measures.
\begin{theorem}\cite[Corollary 1.3]{Bosh}\label{theo:boshcond}
Let $(X,S)$ be a minimal subshift and let $p_X$ denote the factor complexity function of $X$. If
$$\liminf_{n\to\infty}\frac{p_X(n)}{n}<\alpha<+\infty,$$
then the number of ergodic invariant probability measures of $(X,S)$ is at most $\max(\lfloor\alpha\rfloor,1)$, where $\lfloor\alpha\rfloor$ denotes the greatest integer which is smaller than $\alpha$.
\end{theorem} 

Cyr and Kra showed that any subshift (not necessarily minimal) with linear factor complexity has at most finitely many nonatomic ergodic invariant probability measures, and so at most countably many ergodic invariant probability measures \cite{cyrkra1}. In \cite{cyrkra2}, the authors show that it is possible to realize any arbitrarily low superlinear factor complexity with a minimal subshift having uncountably many ergodic invariant probability measures. More precisely, they prove the following theorem.
\begin{theorem}\cite[Theorem 1.1]{cyrkra2}\label{theo:ck2}
If $(p_n)_{n\in\N}$ is a sequence of natural numbers such that 
$$\liminf_{n\to\infty}\frac{p_n}{n}=\infty,$$ 
then there exists a minimal subshift $(X,S)$ which supports uncountably many ergodic measures and is such that 
$$\liminf_{n\to\infty}\frac{p_X(n)}{p_n}=0.$$
\end{theorem}
\begin{Rem}
Theorem \ref{theo:main3} is a strengthening of Theorem \ref{theo:ck2} in two senses. The first one is that the conclusion is about the limit of $\frac{p_X(n)}{p_n}$ and not only about the liminf. Here it is worth mentioning that the behaviour of the complexity along different sequences could be very different, for instance, examples in \cite[Section 4]{DDMP16} show that it is possible to have $\liminf p_X(n)/p_n =0$ and $\limsup p_X(n)/p_n=+\infty$ for a superlinear sequence $(p_n)_{n\in \N}$.

The second strengthening is that we are able to realize not only uncountably many invariant ergodic measures, but any Choquet simplex of invariant measures, which at the same time implies that \cref{theo:main3} is a generalization of \cref{theo:do91}.
\end{Rem}
\subsection{Inductive limits.}\label{subsec:inductive}
For a sequence of group homomorphisms $(h_i\colon G_i\to G_{i+1})_{i\in\N}$, the associated {\it inductive limit} is defined by
$$\varinjlim_i(G_i,h_i)=\{(g,i)\in G_i\times \N\}/\sim,$$
where $(g,i)\sim (g',i')$ if and only if there exists $k\geq i,i'$ such that $h_k\circ\cdots\circ  h_i(g)=h_k\circ \cdots \circ h_{i'}(g')$. We denote by $[g,i]$ the equivalence class of the element $(g,i)$. The inductive limit $\varinjlim\limits_{i}(G_i,h_i)$ is also denoted
$$ G_0 \xrightarrow[]{h_0} G_1 \xrightarrow[]{h_1} G_2 \xrightarrow[]{h_2} \cdots, $$
and it has a group structure given by the product 
$$[g,i]\cdot [g',i']= [h_k\circ \cdots \circ h_i(g)\cdot h_k\circ \cdots \circ h_{i'}(g'),k+1],$$ where $k\geq i,i'$.

For a $m\times n$ integer matrix $A$, we consider the natural homomorphism it induces, {\it i.e.}, $A\colon \Z^{m}\to \Z^{n}$ (or $A\colon \QQ^{m}\to \QQ^{n}$ ), $x\mapsto Ax$. 

\subsection{Dimension groups.}\label{subsec:dg}
An {\it ordered group} is a pair $(G,G^+)$ where $G$ is a countable abelian group $G$ and $G^+$ is a subset of $G$, called the {\it positive cone}, satisfying
$$G^+ + G^+ \subset G^+, \quad  G^+ \cap (-G^+) = \{ 0 \}, \quad  G^+ - G^+ = G.
$$
We write $a\leq b$ if $b-a \in G^+$, and $ a<b$ if $b-a \in  G^+$ and $b \neq a$. Two ordered groups $(G,G^+)$ and $(H,H^+)$ are {\it isomorphic} if there exists a group isomorphism $\varphi\colon G\to H$ such that $\varphi(G^+)=H^+$. We also say that $G$ and $H$ are {\it order isomorphic} when the positive cones are clear from the context. An ordered group $(G,G^+)$ is a {\it dimension group} if for every $i\geq 0$ there exists an integer $d_i\geq 1$ and a positive homomorphism $A_i\colon \Z^{d_i}\to \Z^{d_{i+1}}$ such that $(G,G^+)$ is isomorphic to the inductive limit
$$ \Z^{d_0}  \xrightarrow[]{A_0} \Z^{d_1}  \xrightarrow[]{A_1} \Z^{d_2} \xrightarrow[]{A_2} \cdots. $$
endowed with the induced order. 
The dimension group is called {\it simple} if all matrices $A_i$ can be chosen to be positive. An {\it order unit} for $(G,G^+)$  is an element $u $ in $G^+$  such that, for all $a $ in $G$, there exists  some non-negative integer $k$ with $a  \leq ku$. An {\it ordered group with unit} is a triple $(G,G^+,u)$, where $(G,G^+)$ is an ordered group and $u$ is an order unit for $(G,G^+)$. A {\it dimension group with unit} is a triple $(G,G^+,u)$, where $(G,G^+)$ is a dimension group and $u$ is an order unit for $(G,G^+)$. If $(G,G^+)$ is isomorphic to the inductive limit 
$$ \Z^{d_0}  \xrightarrow[]{A_0} \Z^{d_1}  \xrightarrow[]{A_1} \Z^{d_2} \xrightarrow[]{A_2} \cdots $$
endowed with the induced order, any element $[x,i]$, where all the entries of the vector $x\in \Z^{d_i}$ are positive, is an order unit for $(G,G^+)$.
Two ordered groups with unit $(G,G^+,u)$ and $(H,H^+,v)$ are {\em isomorphic} if there exists a group isomorphism $\varphi \colon G \rightarrow H$ such that $\varphi(G^+)=H^+$ and $\varphi (u)=v$.

The following lemmas will be useful in the proofs of the main theorems. Recall that an abelian group $(G,+)$ is called {\it divisible} if for each $g\in G$ and each integer $k\geq 1$, there exists $h\in G$ such that $kh=g$. 

\begin{lemma}\label{prop:divisible2}
Let $(A_i)_{i\in\N}$ be a sequence of $d_{i+1}\times d_i$ integer matrices, let $G=\varinjlim\limits_{i}(\Z^{d_i},A_i)$. If for all $i\geq 1$, all the entries of $A_i$ are divisible by $i+1$, then $G$ is a divisible group.
\end{lemma}
\begin{proof}
We must show that for all $g\in G$, for all positive integer $k$, there exists $h\in G$ such that $kh=g$. Let $g=[x,i]\in G$ and $k$ a positive integer.

If $k=1$, it suffices to take $h=g$.
Suppose that $k>1$. If $i=0$, let $z=A_{k-1}\cdots A_0x$. Then, $[z,k]=[x,0]$. Let $y=\frac{1}{k}A_{k-1}\cdots A_0x$. Since every entry of $A_{k-1}$ is divisible by $k$, $y\in \Z^{d_k}$. Taking $h=[y,k]$, we have that $kh=k[y,k]=[z,k]=[x,0]=g$. If $i>0$, $ki>i$. In this case, let  $z=A_{ki-1}\cdots A_ix$. Since every entry of $A_{ki-1}$ is divisible by $ki$, they are all divisible by $k$ as well. Let $y=\frac{1}{k}A_{ki-1}\cdots A_ix$. As before, $y\in \Z^{d_{ki}}$. Taking $h=[y,ki]$, we have that $kh=k[y,ki]=[z,ki]=[x,i]=g$. 
\end{proof}

\begin{lemma}\label{prop:divisible1}
Let $(A_i)_{i\in\N}$ be a sequence of $d_{i+1}\times d_i$ positive integer matrices, let $G=\varinjlim\limits_{i}(\Z^{d_i},A_i)$ and $H=\varinjlim\limits_{i}(\QQ^{d_i},A_i)$. Let $G^+$ and $H^+$ be the positive cones of $G$ and $H$ respectively, given by the induced order on the inductive limits. Let $u=[{\bf 1}, 0]_G$ and $v=[{\bf 1},0]_H$ be order units for $(G,G^+)$ and $(H,H^+)$ respectively, where $[\cdot,\cdot]_G$ and $[\cdot,\cdot]_H$ denote the equivalence classes modulo $\sim$ in $G$ and $H$ respectively, and ${\bf 1} \in \Z^{d_0}\subseteq \QQ^{d_0}$ is the vector $(1,1,\cdots, 1)$. If $G$ is divisible, then $(G,G^+,u)$ and $(H,H^+,v)$ are isomorphic as ordered groups with unit. 
\end{lemma}

\begin{proof}
Let $\varphi\colon G\to H$ be the map defined by $[x,i]_G\mapsto [x,i]_H$. This map is a well defined injective group homomorphism. By definition, $\varphi([{\bf 1},0]_G)=[{\bf 1},0]_H$. Let $[y,i]_H \in H$, let $(y,i)\in \QQ^{d_i}\times \N$ be a representative. Since $y\in \QQ^{d_i}$, we can write it in the form
$$\left(\frac{w_1}{k_1},\cdots, \frac{w_{d_i}}{k_{d_i}}\right),$$
where $w_1,\cdots, w_{d_i}\in \Z$, $k_1,\cdots, k_{d_i}\in \N^{\ast}$. Let $k=lcm\{k_\ell\}$, so that $y=\frac{1}{k}(w_1',\cdots, w_{d_i}')$, some $w_1',\cdots, w_{d_i}'\in\Z$. Let $z=(w_1',\cdots, w_{d_i}')\in \Z^{d_i}$. Since $G$ is divisible, there exists $[x,j]_G\in G$ such that $k[x,j]_G=[z,i]_G$. This means that there exists $s\geq j,i$ such that $kA_s\cdots A_jx=A_s\cdots A_i z$, and since $ky=z$, we obtain that $[x,j]_H=[y,i]_H$. This proves that $\varphi$ is surjective. It is strightforward to show that $\varphi(G^+)=H^+$. 
\end{proof}

The next lemma concerning dimension groups will be useful in the sequel. We first introduce a definition.
\begin{definition}\label{def:adapted} let $(A_i)_{i\in \N}$ be a sequence of $d_{i+1}\times d_i$ positive integer matrices. We say that the sequence of $d_i\times d_i$ matrices $(J_{i})_{i\in \N}$, is {\it adapted} to $(A_i)_{i\in \N}$ if the following two conditions hold, 
\begin{enumerate}
\item \label{item:adaptedcond1} $J_i, J_i^{-1}$ have positive rational entries and for all $i\in \N$, there exists $M\geq i$ such that $A_{M} A_{M-1}\cdots A_{i} J_{i}^{-1}$ has positive integer entries.
\item \label{item:adaptedcond2} $B_i:=J_{i+1}A_iJ_i^{-1}$ has positive integer entries for all $i\in \N$. 
\end{enumerate}
\end{definition}

\begin{lemma} \label{lema:AdaptedMatrices}
 	Let $(A_i)_{i\in \N}$ be a sequence of positive integer matrices and let $(J_{i})_{i\in \N}$ be  a sequence adapted to $(A_{i})_{i\in \N}$. Let $G=\varinjlim\limits_{i}(\Z^{d_i},A_i)$ and $H=\varinjlim\limits_{i}(\Z^{d_i},B_i)$, where $(B_i)_{i\in\N}$ is the sequence of matrices defined in Definition \ref{def:adapted}-(\ref{item:adaptedcond2}). Consider the positive cones $G^+$ and $H^+$ given by the induced order, and the order units $u=[A_{M_0} A_{M_0-1}\cdots A_{0} J_{0}^{-1}{\bf 1},M_0+1]_G$ and $v=[{\bf 1},0]_H$, where $[\cdot,\cdot]_G$ and $[\cdot,\cdot]_H$ denote the equivalence classes modulo $\sim$ in $G$ and $H$ respectively, ${\bf 1} \in \Z^{d_0}$ is the vector $(1,1,\cdots, 1)$, and $M_0\geq 0$ is such that $A_{M_0} A_{M_0-1}\cdots A_{0} J_{0}^{-1}$ has positive integer entries. Then, $(G,G^+,u)$ and $(H,H^+,v)$ are isomorphic as ordered groups with unit. 
\end{lemma}

\begin{proof} 
Note that condition (\ref{item:adaptedcond2}) allows us to define morphisms $A_i\colon J_{i}^{-1}\Z^{d_i}\to J_{i+1}^{-1}\Z^{d_{i+1}}$. Note also that $H$ is order isomorphic to the inductive limit $K=\varinjlim\limits_{i}(J_i^{-1}\Z^{d_i},A_i)$, via the map $\varphi\colon H\to K$, $[x,i]_H\mapsto [J_i^{-1}x,i]_K$. This order isomorphism maps the unit $v$ into $[J_0^{-1}{\bf 1},0]_K$. On the other hand, condition (\ref{item:adaptedcond1}) ensures that $K$ is order isomorphic to $G$ via the map $\theta\colon K\to G$, $[y,i]_K\mapsto [A_{M} A_{M-1}\cdots A_{i} J_{i}^{-1}y, M+1]_G$, where $M\geq i$ is such that $A_{M} A_{M-1}\cdots A_{i} J_{i}^{-1}$ has positive integer entries. We conclude by noticing that $\theta\circ\varphi(H^+)=G^+$ and $\theta\circ\varphi(v)=u$. 
\end{proof}

Given a dimension group with unit $(G,G^+,u)$, a {\it trace} of $(G,G^+,u)$ is a group homomorphism $p\colon G\to \mathbb{R}$ such that $p$ is non-negative ($p(G^+)\geq 0$) and $p(u)=1$. The collection of all traces of $(G,G^+,1)$ is denoted by $S(G,G^+,u)$. An element $g \in G$ is said to be {\it infinitesimal} if $p(g) = 0$ for every trace $p \in S(G, G^+,u)$. The collection of all infinitesimals of $(G,G^+,u)$ form a subgroup of $G$, called the infinitesimal subgroup of $(G,G, u)$ and denoted $\Inf(G,G^+,u)$.

The following result will be used in the proof of Theorem \ref{theo:main3}. It connects simple dimension groups and Choquet simplices and its proof can be found in \cite{effros}. We use the formulation of \cite{malen-hoynes}.
\begin{theorem}\cite[Theorem 3.22]{malen-hoynes}\label{theo:simplexanddg}
Let $K$ be a Choquet simplex and $\Aff(K)$ the additive group of real, affine and continuous functions on $K$. Let $H$ be a countable dense subgroup of $\Aff(K)$, and suppose there exists a torsion-free abelian group $G$ and a homomorphism $\theta:G\to H$. Then, letting
$$G^+=\{g\in G \mid \theta(g)(p)>0 \forall p\in K\}\cup \{0\},$$
we get that $(G,G^+)$ is a simple dimension group such that $\Inf(G)=\ker(\theta)$. In particular, if $G=H$ (with $\theta$ the identity map) and $G$ contains the constant function $1$, then $\Inf(G)=\{0\}$ and $S(G,G^+,1)$ is affine homeomorphic to $K$ by the map sending $k\in K$ to $\hat{k}:G\to \R$, where $\hat{k}(g)=g(k)$ for all $g\in G$.
\end{theorem}

Let $(X,T)$ be a Cantor minimal system. The {\it dynamical dimension group} of $(X,T)$ or simply the {\it dimension group} of $(X,T)$ is the following triple,
$$K^0(X,T)=(H(X,T), H^+(X,T), [1]),$$
where $H(X,T)=C(X,\Z)/\beta C(X,\Z)$, $\beta$ is the coboundary map, $[\cdot]$ denotes the class modulo $\beta C(X,\Z)$ of an element in $H(X,T)$, $H^+(X,T)$ is the set of classes of non-negative functions and $1$ is the constant function equal to $1$.
\begin{theorem}\cite[Corollary 6.3]{HPS}\label{theo:isdg}
If $(X,T)$ is a Cantor minimal system, the triple $K^0(X,T)$ is a simple dimension group with unit. Furthermore, if $(G,G^+,u)$ is a simple dimension group with unit, then there exists a Cantor minimal system $(X,T)$ such that $K^0(X,T)$ is isomorphic to $(G,G^+,u)$ as ordered groups with unit.   
\end{theorem}
Given an invariant measure $\mu \in \cM(X,T)$, we define the trace $\tau_{\mu}$ on $K^0 (X,T)$  by  $\tau_{\mu} ([f]):= \int f d\mu$. The correspondence $\mu\mapsto \tau_\mu$ is an affine isomorphism from $\cM(X,T)$ to $S(K^0(X,T))$, so that traces of the dynamical dimension group $K^0(X,T)$ correspond to the Choquet simplex of invariant measures of the system $(X,T)$ (see for instance \cite[Theorem 5.5]{HPS}).

\subsection{Orbit equivalence.}\label{subsec:oe}
Two minimal dynamical systems $(X_1,T_1)$ and $(X_2,T_2)$ are {\it orbit equivalent} if there exists a homeomorphism $\phi\colon X_1 \to X_2$ sending orbits of the $T_1$-action onto orbits of the $T_2$-action, {\it i.e.},
$$\phi( \{ T_1^k (x)  : k \in \Z \})= \{ T_2^k \phi (x) : k \in \Z\} \text{ for all } x \in X_1.$$
Orbit equivalence implies the existence of two maps $n_1\colon X_1  \to \Z$ and $n_2\colon X_2 \to \Z$ (uniquely defined by minimality) such that, for all $x \in X_1$,
$$\phi \circ T_1 (x)= T_2^{n_1(x)} \circ \phi(x) \mbox { and }  \phi \circ T_1^{n_2(x)} (x)= T_2 \circ \phi (x).$$
The dynamical systems $(X_1,T_1)$ and $(X_2,T_2)$ {\it strong orbit equivalent}  if $n_1$ and $n_2$ both have at most one point  of discontinuity. Such notion is natural since it was shown  in \cite{Boyle} that if $n_1$ (or $n_2$) is continuous, then the two systems are {\it flip conjugate}, that is, $(X_1,T_1)$ is either conjugate to $(X_2,T_2)$ or to its inverse $(X_2,T_2^{-1})$.

The following result is one of the most important theorems connecting dimension groups and orbit equivalence.
\begin{theorem}\cite[Theorem 2.1]{GPS}\label{theo:caracterizationoe}
Let $(X_1,T_1)$ and $(X_2,T_2)$ two Cantor minimal systems. Then $(X_1,T_1)$ and $(X_2,T_2)$ are strong orbit equivalent if and only if $K^0(X_1,T_1)$ and $K^0(X_2,T_2)$ are isomorphic as ordered groups with unit.
\end{theorem}

\section{Bratteli diagrams and $\cS$-adic subshifts}\label{sec:brattelidef}

\subsection{Bratteli diagrams and Bratteli-Vershik systems.}\label{subsec:bratteli}
A {\it Bratteli diagram} is an infinite directed graph $(V,E)$ where the set of vertices $V$ and the set of edges $E$ can be written as a countable disjoint union of non-empty finite sets,
$$V=V_0\cup V_1\cup V_2\cup\cdots \quad \mbox{ and } \quad E=E_1\cup E_2\cup\cdots,$$
with the property that $V_0$ is a single point and there exist a {\it range map} $r\colon E\to V$ and a {\it source map} $s\colon E\to V$ so that $r(E_i)\subseteq V_i$ and $s(E_i)\subseteq V_{i-1}$. Also, we assume that $s^{-1}(v)\neq \emptyset$ for all $v\in V$ and $r^{-1}(v)\neq \emptyset$ for all $v\in V\setminus V_0$.

For $i\in\N$, let $A_i$ denote the $i$th {\it incidence matrix} of $(V,E)$, that is, the $|V_{i+1}|\times|V_i|$ matrix whose coefficient $A_i(k,j)$ is the number of edges connecting $u_j\in V_i$ with $v_k\in V_{i+1}$.

Let $(\ell_i)_{i\in\N}$ be a subsequence of $\N$. Define a new Bratteli diagram $(V',E')$ by letting $V_i'=V_{\ell_i}$ and $A_i'=A_{\ell_{i+1}-1}A_{\ell_{i+1}-2}\cdots A_{\ell_i}$. The sets of edges $E_i'$ and the range and source maps are obtained from the new incidence matrices. The diagram $(V',E')$ is called a {\it telescoping} of $(V,E)$ to $(\ell_i)_{i\in\N}$. The sequence $(\ell_i)_{i\in\N}$ is called a {\it sequence of telescoping depths}. If a Bratteli diagram $(V',E')$ can be telescoped to obtain $(V,E)$, then we say that $(V',E')$ is a {\it microscoping} of $(V,E)$. A Bratteli diagram is {\it simple} if it can be telescoped so that all its incidence matrices have only positive entries. 

The following technical lemma will be used to the proof of Theorem \ref{theo:main1}.

\begin{lemma}\label{lem:h's}
Let $(V,E)$ be a simple Bratteli diagram and let $p_n$
be a sequence of positive real numbers such that $\lim n/p_n = 0$. Then, there exists a telescoping $(V',E')$ of $(V,E)$ and a sequence of positive integers $(h_i)_{i\in\N}$ satisfying the following conditions:
\begin{enumerate}
    \item \label{item:h'slemmacond1} for all $i\in\N$, $h_i+h_i^2<\min_{k,j}\{(A_i')_{k,j}\}$;
    \item \label{item:h'slemmacond2} for all $i\geq 1$,
    $$\frac{|V_i'|^3 t}{p_{t}} <\frac{1}{i},$$
    for all $t\geq h_i$;
    \item \label{item:h'slemmacond3} for all $i\geq 1$, $h_i>2|V_{i-1}'|$.
\end{enumerate}
\end{lemma}
\begin{proof}
We will define a sequence of telescoping depths $(\ell_i)_{i\in\N}$ and the sequence $(h_i)_{i\in\N}$ inductively. The aforementioned diagram $(V',E')$ will be the telescoping of $(V,E)$ to $(\ell_i)_{i\in\N}$.

For $i=0$, set $\ell_0=0$ and $h_0=1$. Let $\ell_1$ be a positive integer such that $h_0+h_0^2<\min_{k,j}\{(A_{\ell_1-1}A_{\ell_1-2}\cdots A_0)_{k,j}\}$.  We can always find such a number thanks to the simplicity of $(V,E)$. Assume that for $i\geq 1$ we have defined $\ell_0,\ell_1,\cdots, \ell_i$ and $h_0,h_1,\ldots,h_{i-1}$. Since $\lim n/p_n = 0$, we may take a positive integer $h_i$ such that for all $t\geq h_i$, 
$$\frac{t}{p_t}<\frac{1}{i|V_{\ell_i}|^3}.$$
Obviously we can also assume that $h_i>2|V_{\ell_{i-1}}|$. We set $\ell_{i+1}$ to be a positive integer such that $\ell_{i+1} >\ell_i$ and  
$$h_i+h_i^2<\min_{k,j}\{(A_{\ell_{i+1}-1}A_{\ell_{i+1}-2}\cdots A_{\ell_i})_{k,j}\}.$$
We have then defined $\ell_{i+1}$ and $h_{i}$.  By construction, the telescoping of $(V,E)$ to the levels $(\ell_i)_{i\in\N}$ satisfies conditions (\ref{item:h'slemmacond1})--(\ref{item:h'slemmacond3}).
\end{proof}

An {\it ordered Bratteli diagram} is a Bratteli diagram together with a linear ordering on $r^{-1}(v)$ for each $v\in V\setminus V_0$. This defines a partial order $\geq$ on $E$. An ordered Bratteli diagram $(V,E)$ together with a partial order $\geq$ on $E$ is denoted $(V,E,\geq)$. Let $E_{min}$ and $E_{max}$ denote the sets of minimal and maximal edges, respectively. An {\it infinite path} in $(V,E)$, that is, a sequence of the form $(e_1,e_2,\cdots)$ where $e_i\in E_i$ and  $r(e_i)=s(e_{i+1})$ for all $i\in\N$, is said to be {\it minimal} (resp. {\it maximal}) if all its edges belong to $E_{min}$ (resp. $E_{max}$). An ordered Bratteli diagram is {\it properly ordered} if it is simple and it has a unique minimal path and a unique maximal path, which we denote by $x_{min}$ and $x_{max}$, respectively. An ordered Bratteli diagram is {\it left/right ordered} if whenever two edges $e_1$, $e_2$ with the same range verify $e_1\geq e_2$ and $s(e_1)=v_i$, $s(e_2)=v_j$, we have $i\geq j$.

Given a properly ordered Bratteli diagram $(V,E,\geq)$, it is possible to define a dynamic on it as follows: let $X_B$ denote the space of infinite paths on $E$ (where we assume that $X_B$ is infinite). We endow $X_B$ with a topology by giving a basis of open sets, namely the family of {\it cylinder sets},
$$[e_1 , e_2 , \cdots, e_k ]_B = \{(f_1 , f_2 , \cdots) \in X_B: f_i=e_i \mbox{ for all } 1\leq i\leq k\}.$$
The cylinder sets are also closed. The space $X_B$ endowed with this topology is called the {\it Bratteli compactum} associated with $(V,E)$ and it is a Cantor space. We define the {\it Vershik map} $V_B$ on $X_B$ as follows: $V_B(x_{max})=x_{min}$; if $x=(e_1,e_2,\cdots)$ is not the maximal path, let $k$ be the smallest integer such that $e_k$ is not a maximal edge, let $f_k$ be the sucesor of $e_k$ on $E_k$, and define $V_B(x)=(f_1,f_2,\cdots, f_{k-1},f_k,e_{k+1},e_{k+2},\cdots)$, where $(f_1,\cdots, f_{k-1})$ is the minimal finite path on $E_1\circ E_2\circ \cdots \circ E_{k-1}$ with range equal to $s(f_k)$. The system $(X_B,V_B)$ is called the {\it Bratteli-Vershik dynamical system} associated with $(V,E,\geq)$.

It was shown by Herman, Putnam and Skau \cite{HPS} that for any Cantor minimal system $(X,T)$, there exists a properly ordered Bratteli diagram $(V,E,\geq)$ such that $(X,T)$ and $(X_B,V_B)$ are conjugate. In this case, $(V,E,\geq)$ is called a Bratteli-Vershik {\it model} for $(X,T)$.

To a Bratteli diagram $(V,E)$ is associated a dimension group with unit denoted $K_0(V,E)$. By definition, $K_0(V,E)$ is the inductive limit of the system
$$ \Z^{|V_0|}  \xrightarrow[]{A_0} \Z^{|V_1|}   \xrightarrow[]{A_1} \Z^{|V_2|}\xrightarrow[]{A_2} \cdots, $$
endowed with the induced order, where $(A_i)_{i\in\N}$ is the sequence of incidence matrices of $(V,E)$. The order unit in $K_0(V,E)$ is the element $[1,0]\in \varinjlim\limits_{i}(\Z^{|V_i|},A_i)$. Two Bratteli diagrams $(V,E)$ and $(V',E')$ have isomorphic dimension groups with unit if and only if $(V',E')$ can be obtained from $(V,E)$ by a finite number of telescopings and microscopings. If $(V,E,\geq)$ is a properly ordered Bratteli diagram and $(X_B,V_B)$ is its associated Bratteli-Vershik system, then one has $K_0(V,E)\cong K^0(X_B,V_B)$ as ordered groups with unit.

A simple Bratteli diagram $(V,E)$ has the {\it equal path number property} if for all $i\geq 1$, for all $u,v\in V_i$, $|r^{-1}(u)|=|r^{-1}(v)|$. Equivalently, $(V,E)$ has the equal path number property if for all $i\geq 1$, for all $u,v\in V_i$, the number of finite paths from $V_0$ to $u$ and from $V_0$ to $v$ are equal. In terms of the matrices, this means that for all $i\in\N$, the sum of the entries on a row of $A_i$ is constant.  This property is called the {\it equal row sum (ERS)} property of the matrix $A_i$. The equal path number property was introduced in \cite{noruegos}, where the authors prove the following result.  
\begin{theorem}\cite[Theorem 8]{noruegos}\label{teo:ers}
The family of expansive Bratteli-Vershik systems associated with Bratteli diagrams with the equal path number property coincides with the family of Toeplitz flows up to conjugacy.  
\end{theorem}

\subsection{Morphisms and $\cS$-adic systems}\label{subec:Sadic}
\noindent Let $\cA$, $\cB$ be two finite alphabets. Let $\tau \colon \cA^{\ast} \to \cB^{\ast}$ be a morphism. We say that $\tau$ is {\it non-erasing} if the image of any letter is a non-empty word. The {\it incidence matrix} of $\tau$ is the $|\cB|\times|\cA|$ matrix whose entry at a position $(b,a)$ is the number of times that $b$ appears in $\tau(a)$. For $a\in\cA$, the length of the word $\tau(a)\in \cB^{\ast}$ is denoted $|\tau(a)|$. The morphism $\tau$ is {\em positive} if the entries of $|\cB|\times|\cA|$ are all positive and is {\it left proper} (resp. {\it right proper}) if there exists a letter $b\in\cB$ such that for every $a\in\cA$, $\tau(a)$ starts with $b$ (resp. ends with $b$); it is {\it proper} if is both left and right proper. We say that $\tau$ is a {\it hat morphism} if for all $a,b\in \cA$, the letters appearing in $\tau(a)$ and $\tau(b)$ are all distinct. By concatenation, a morphism $\tau\colon \cA^{\ast}\to \cB^{\ast}$ can be extended to $\cA^\N$ and $\cA^\Z$.\\

For a non-erasing morphism $\tau \colon \cA^{\ast} \to \cB^{\ast}$, define $\|\tau\|=\max_{a\in \cA}\{|\tau(a)|\}$ and $\langle \tau \rangle= \min_{a\in \cA}\{|\tau(a)|\}$.
The following notion was introduced in \cite[Section 6]{DDMP20} in the context of finite topological rank minimal subshifts. Assume that $\cB=\{b_1,\ldots,b_{|\cB|}\}$ and for $a\in \cA$, write $\tau(a)=b_{i_1}^{\ell_1}b_{i_2}^{\ell_2}\cdots b_{i_{k(a)}}^{\ell_{k(a)}}$, for some $i_1,\ldots,i_{k(a)}\in \{1,\ldots,|\cB|\}$ and where $b_{i_{j+1}}\neq b_{i_j}$ for all $j=1,\ldots,k(a)-1$. The integer $k(a)$ represents the number of times one needs to switch letters while writing $\tau(a)$, plus one.  
The {\it repetition complexity} of $\tau$ is defined as
\[\textrm{r-comp}(\tau)= \sum_{a\in \cA} k(a). \] 

Note that if $\tau$ is positive then $k(a)\geq |\cB|$ for all $a\in \cA$ and $\textrm{r-comp}(\tau)\geq |\cA||\cB|$. For a morphism $\tau\colon \cA\to \cB$, we let $\tau(\cA^{\Z})$ denote the smallest subshift containing $\tau((x_i)_{i\in \Z})$ for all $(x_i)_{i\in \Z}\in \cA^{\Z}$. 

The following result is a consequence of Lemma 6.10 and Theorem 6.11 in \cite{DDMP20}.

\begin{theorem}\label{theo:boundedrepetition} 
Let $\sigma_1 \colon\cC^{\ast} \to \cD^{\ast}$, $\sigma_2 \colon  \cB^{\ast} \to \cC^{\ast}$, and $\sigma_3 \colon  \cA^{\ast} \to \cB^{\ast}$ be three morphisms and assume that $\sigma_2$ and $\sigma_3$ are positive.
Then we have that
\[{ p_{\sigma_1 \circ \sigma_2\circ \sigma_3 (\cA^{\Z}) } (n) \leq \begin{cases} (|\cC| + (|\cB|+1)\textrm{r-comp}(\sigma_2))n  &  \text{ if } n \in I \\ (|\cB|+|\cC|+\textrm{r-comp}(\sigma_2) + (|\cA|+1)\textrm{r-comp}(\sigma_3))  n & \text{ if } n \in J \end{cases} }\]
where $I=[\| \sigma_1 \| , \langle \sigma_1 \circ \sigma_2\rangle)$ and $J=[\langle  \sigma_1 \circ \sigma_2  \rangle , \|\sigma_1 \circ \sigma_2\|)$.
\end{theorem}

We briefly recall the definition of $\cS$-adic subshifts (see \cite{BSTY} for more details). A {\it directive sequence} ${\bt}=(\tau_i)_{i\geq 0}$ is a sequence of non-erasing morphisms $\tau_i\colon \cA_{i+1}^{\ast}\to \cA_i^{\ast}$, $i\in \N$. We let $\tau_{[i,k)}$ denote the composition $\tau_i\circ\tau_{i+1}\circ\cdots\circ \tau_{k-1}$. We say that $\bs$ is {\it everywhere growing} if $\langle \tau_{[0,i)} \rangle=\min_{a\in \cA_i}\{|\tau_{[0,i)}(a)|\}$ tends to $\infty$ as $i\to\infty$. We say that $\bs$ is {\it primitive} if for every $i\geq 0$, there exists $k\geq i$ such that $\tau_{[i,k)}$ has a positive incidence matrix.

For $i\geq 0$, the {\it language of order $i$} $L_{{\bs}}^{(i)}$ associated with ${\bs}$ is defined as
$$L_{{\bs}}^{(i)}=\{w\in\cA_i^{\ast}: \exists k>i, \exists a\in \cA_k, w\prec \tau_{[i,k)}(a)\}.$$
For each $i\geq 0$, the set $X_{{\bs}}^{(i)}$ is the set of infinite words $x\in\cA_i^\Z$  whose factors belong to $L_{{\bs}}^{(i)}$. 
We set $X_{{\bs}}=X_{{\bs}}^{(0)}$, $L_{{\bs}}=L_{{\bs}}^{(0)}$ and call $(X_{{\bs}}, S)$ the {\it $\cS$-adic system generated by the directive sequence} ${\bs}$, where $S$ is the shift transformation. For all $\ell \geq 1$, we denote by $L_{{\bs},\ell}^{(i)}$ the subset of length $\ell$ factors of $L_{{\bs}}^{(i)}$. When the directive sequence $\bs$ is primitive, $(X_{\bs},S)$ is a minimal subshift.
 
As a consequence of Theorem \ref{theo:boundedrepetition}, we get the following result. 

\begin{proposition} \label{prop:Complexitybound}
Let $(X_{\bs},S)$ be the $\cS$-adic system generated by the directive sequence $\bs = (\tau_i \colon  \mathcal{A}_{i+1}^{\ast}\to \mathcal{A}_i^{\ast})_{i\geq 0}$. For $n\in \N, n\geq \|\tau_0\|$, let $i=i(n)$ such that 
$n\in [\|\tau_{[0,i)}\|, \|\tau_{[0,i+1)} \| )$.
Then 
\[ p_{X_{\bs}} (n) \leq \begin{cases} (|\cA_{i}| + (|\cA_{i+1}|+1)\textrm{r-comp}(\tau_i))n  & \text{ if }  n \in  I_i \\ (|\cA_{i+1}|+|\cA_{i}|+\textrm{r-comp}(\tau_i) + (|\cA_{i+2}|+1)\textrm{r-comp}(\tau_{i+1}))  n & \text{ if } n \in J_i \end{cases} \]
where $I_i=[\| \tau_{[0,i)} \| , \langle \tau_{[0,i+1)}\rangle)  $ and $J_i= [\langle  \tau_{[0,i+1)} \rangle , \|\tau_{[0,i+1)}\|) $. 

\end{proposition}

\begin{proof}
It follows from applying \cref{theo:boundedrepetition} for the morphisms $\sigma_3=\tau_{i+1}\colon\cA_{i+2}^{\ast}\to \cA_{i+1}^{\ast}$, $\sigma_2=\tau_{i}\colon\cA_{i+1}^{\ast}\to \cA_{i}^{\ast}$ and $\sigma_1=\tau_{[0,i)}\colon\cA_{i}^{\ast}\to \cA_{0}^{\ast}$, and recalling that $X_{\bs}\subseteq \sigma_1\circ \sigma_2 \circ \sigma_3(\cA_{i+2}^{\Z})$.
\end{proof}

The following notion connecting ordered Bratteli diagrams and $\cS$-adic systems was originally  introduced in \cite{fdurand}. Given an ordered Bratteli diagram $(V,E,\geq)$, let $i\geq 1$ and consider $V_i$, $V_{i+1}$ as finite alphabets. For every $u\in V_{i+1}$, consider the ordered list $(e_1,e_2,\cdots, e_k)$ of edges on $E_{i+1}$ arriving to $u$, and let $(v_1,v_2,\cdots, v_k)$ the list of labels of the sources of these edges in $V_i$. This defines a morphism $\tau_i\colon V_{i+1}^{\ast} \to V_{i}^{\ast}$,  $u\mapsto v_1v_2\cdots v_k$. For $i=0$, we let $\tau_0\colon V_1^{\ast} \to E_1^{\ast}$ denote the morphism such that  
$\tau(v)=e_1(v)\ldots e_\ell(v)$, where  $e_1(v),\ldots,e_\ell(v)$ are the edges connecting $V_0$ with $v$ according to the order of the diagram. For $i\geq 0$ we say that $\tau_i$ is {\it the morphism read on $(V,E,\geq)$ at level $i$} and that the directive sequence ${\boldsymbol \tau}=(\tau_i)_{i\geq 0} $ is the {\it sequence of morphisms read  on} $(V,E,\geq)$. Note that the morphism read on $(V,E,\geq)$ at level $0$ is a hat morphism.\\

The previous notion gives an $\cS$-adic subshift naturally associated with each Bratteli-Vershik system. The following result, which is a slight modification of \cite[Proposition 4.5]{DDMP20}, states when these two representations are conjugate as dynamical systems (see also \cite[Proposition 2.2]{DurandLeroy2012}).
\begin{proposition}\cite[Proposition 4.5]{DDMP20}\label{prop:isom}
Let $(X_{\bs},S)$ be the minimal $\cS$-adic subshift defined by the directive sequence $\bs  = (\tau_i)_{i\geq 0}$, where $\tau_0$ is a hat morphism and $\tau_i$ is proper for all $i\geq 1$. Suppose that all morphisms $\tau_i$ extend by concatenation to a one-to-one map from $X^{(i+1)}_{\bs}$ to $X^{(i)}_{\bs}$. Then, $(X_{\bs},S)$ is conjugate to the Bratteli-Vershik system $(X_{B} , V_{B})$ associated with the ordered Bratteli diagram $(V,E,\geq)$, where $\bs$ is the sequence of morphisms read  on $(V,E,\geq)$.
\end{proposition}
In the following lemma we give sufficient conditions for a simple Bratteli diagram to have an order such that the sequence of morphisms read on the ordered diagram satisfies the hypothesis of \cref{prop:isom}.

\begin{lemma}\label{lema:injective}
Let $(V,E)$ be a simple Bratteli diagram with sequence of incidence matrices $(A_i)_{i\in\N}$. For all $i\in\N$, let $V_i$ be the set of vertices at level $i$ and $m_i=|V_i|$. Suppose that for all $i\geq 1$, $A_i(j,1)> m_{i+1}$ for all $1\leq j\leq m_{i+1}$. Then it is possible to give $(V,E)$ an order $\geq$ such that the sequence $\boldsymbol{\tau}=(\tau_i)_{i\in\N}$ of morphisms read on $(V,E,\geq)$ verifies that for all $i\geq 1$, $\tau_i:V_{i+1}^{\ast}\to V_i^{\ast}$ is proper, injective and extends by concatenation to a one-to-one map from $X^{(i+1)}_{\boldsymbol{\tau}}$ to $X^{(i)}_{\boldsymbol{\tau}}$.
\end{lemma}

\begin{proof}
For the edges of $E_1$, choose any order. Let $i\geq 1$, let $V_i=\{u_1,\cdots, u_{m_i}\}$ and $V_{i+1}=\{v_1,\cdots, v_{m_{i+1}}\}$. For every $v_j\in V_{i+1}$, we order the edges $r^{-1}(v_j)$ in the following manner. The first $j$ edges have source $u_1$ (this is possible since we assume that $A_i(j,1)> m_{i+1}$ for all $1\leq j\leq m_{i+1}$). The next $A_i(j,2)$ edges have source $u_2$. The remaining edges are ordered left/right.

The previous order gives the following morphism $\tau_i\colon V_{i+1}^{\ast}\to V_i^{\ast}$, 
$$\tau_i(v_j)=u_1^{\ell_{1,j}}u_2^{\ell_{2,j}}u_1^{\ell'_{1,j}}u_3^{\ell_{3,j}}u_4^{\ell_{4,j}}\cdots u_{m_i}^{\ell_{m_i,j}} \quad \forall 1\leq j\leq m_{i+1},$$
where $\ell_{1,j}=j,\ell'_{1,j}=A_i(j,1)-j$ and $\ell_{k,j}=A_i(j,k)$ for all $2\leq k\leq m_i$. It is clear that $\tau_i$ is proper. Since $\ell_{1,j}$ is different for each $v_j\in V_{i+1}$, $\tau_i$ is injective on $V_{i+1}$.

Let us show that $\tau_i$ extends by concatenation to a one-to-one map from $X^{(i+1)}_{\boldsymbol{\tau}}$ to $X^{(i)}_{\boldsymbol{\tau}}$. For $x\in X^{(i+1)}_{\boldsymbol{\tau}}$, we define the set $C_{\tau_i}(x)$ of {\it cutting points} in $\tau_i(x)$ as the following set,
$$C_{\tau_i}(x)=\{|\tau_i(x_{[0,\ell)})|:\ell >0\}\cup\{0\}\cup \{-|\tau_i(x_{[\ell,0)})|:\ell <0\}.$$
Note that for the order we have considered, for any word $x\in X^{(i+1)}_{\boldsymbol{\tau}}$, the cutting points of $\tau_i(x)$ are located exactly in those places where a letter $u_{m_i}$ is followed by a letter $u_1$.

Let $x,x'\in X^{(i+1)}_{\boldsymbol{\tau}}$ be two infinite words such that $\tau_i(x)=\tau_i(x')$. We want to prove that $x=x'$. Since $\tau_i(x)=\tau_i(x')$, each time we see the word $u_{m_i}u_1$ in $\tau_i(x)$, we see it in $\tau_i(x')$ as well, and {\it vice versa}, so the cutting points of $x$ and $x'$ are located exactly in the same places. Therefore $\tau_i(x_\ell)=\tau_i(x'_\ell)$ for all $\ell\in\Z$. Since $\tau_i$ is injective in $V_{i+1}$, we conclude that $x_\ell=x'_\ell$ for all $\ell\in\Z$, thus $x=x'$.
\end{proof}

\section{Proof of the main results.}\label{sec:proofs}

\subsection{Building a suitable representation}
Given a dimension group with unit, the main idea is to realize it as the dimension group of a  Bratteli-Vershik system, in which the number of paths between the level $i$ and the root grows much faster than the number of vertices at the level $i$  (see for instance Proposition \ref{prop:Complexitybound}). The rate at which these quantities differ will depend on how little (superlinear) complexity we intend to get. One way to obtain such a representation is to modify a Bratteli-Vershik system into another one of the same strong orbit equivalence class, in which the corresponding matrices have entries much larger than their dimensions. To achieve this, we introduce the following notion/procedure that allows us to split the levels, and then conveniently factorize the matrices.  

\subsubsection*{Splitting the levels}
Let $A$ be a $n\times m$ integer matrix with positive entries and  $d$ a positive integer. For each entry $a_{k,j}$ of $A$, write $a_{k,j}=dq_{k,j}+r_{k,j}$ where $q_{k,j}\in \N$ and $0\leq r_{k,j}< d$. Let $Q=(q_{k,j})_{k,j}$ and $R=(r_{k,j})_{k,j}$. Then 
\[ A=dQ+R.\]
Let $B$ be the $n\times 2m$ matrix such that $B_{\bullet,2i-1}=dQ_{\bullet,i}$ and $B_{\bullet,2i}=R_{\bullet,i}$ for $1\leq i \leq m$. That is, in $B$ we write the columns of $dQ$ and $R$ interleaved.
Notice that \[  A= B C \]
where $C$ is the $2m\times m$ matrix given by $C_{\bullet,i}=e_{2i-1}+e_{2i}$, for $1\leq i \leq m$, $e_i$ being the $i$th canonical vector of $\R^{2m}$. For convenience we write $B=B(A,d)$, $C=C(A,d)$ to stress the dependence on $A$ and $d$.

\begin{proof}[Proof of \cref{theo:main1}]

Let $(G,G^+,u)$ be a simple dimension group with unit, let $(V,E)$ be a Bratteli diagram such that $(G,G^+,u)\cong K_0(V,E)$. Let $(A_i)_{i\in \N}$ denote the sequence of adjacency matrices of $(V,E)$, where $A_i$ is of $n_i\times m_i$ (so $n_i=m_{i+1}$). Telescoping the diagram if needed,  we may assume that the entries of each $A_i$ are positive and thanks to Lemma \ref{lem:h's} we can assume that there exists a sequence of positive integers $(h_i)_{i\in\N}$ such that for all $i\in\N$, $h_i+h_i^2<\min_{k,j}\{(A_i)_{k,j}\}$ and for all $i\geq 1$, 
\begin{equation} \label{equation:growthh's} \frac{m_i^3 t}{p_{t}} <\frac{1}{i}\end{equation}
whenever $t\geq h_i$.
We can also assume that  $h_i>2|m_{i-1}|$ for all $i\geq 1$.
Define $B_i:=B(A_i,h_i)$, $C_i:=C(A_i,h_i)$ for $i\in \N$. The splitting procedure of $(V,E)$ is schematically depicted in Figure \ref{fig:splitting}.

\begin{figure}[h]

\begin{minipage}{.45\textwidth}
\begin{center}
\begin{tikzpicture}[scale=0.4]

\fill[black] (-4,0) circle (5pt);
\fill[black] (4,0) circle (5pt);
\fill[black] (-4,5) circle (5pt);
\fill[black] (4,5) circle (5pt);

\filldraw[color=black, fill=white, thick](-6,-2) circle (4pt);
\filldraw[color=black, fill=white, thick](-2,-2) circle (4pt);
\filldraw[color=black, fill=white, thick](2,-2) circle (4pt);
\filldraw[color=black, fill=white, thick](6,-2) circle (4pt);

\filldraw[color=black, fill=white, thick](-6,3) circle (4pt);
\filldraw[color=black, fill=white, thick](-2,3) circle (4pt);
\filldraw[color=black, fill=white, thick](2,3) circle (4pt);
\filldraw[color=black, fill=white, thick](6,3) circle (4pt);

\draw[purple, dashed, thick] (-4,0) -- (-6,-2);
\draw[purple, dashed, thick] (-4,0) -- (-2,-2);
\draw[purple, dashed, thick] (4,0) -- (2,-2);
\draw[purple, dashed, thick] (4,0) -- (6,-2);

\draw[purple, dashed, thick] (-4,5) -- (-6,3);
\draw[purple, dashed, thick] (-4,5) -- (-2,3);
\draw[purple, dashed, thick] (4,5) -- (2,3);
\draw[purple, dashed, thick] (4,5) -- (6,3);

\draw (-3.85,0) -- (-3.85,5);
\draw (-3.95,0) -- (-3.95,5);
\draw (-4.05,0) -- (-4.05,5);
\draw (-4.15,0) -- (-4.15,5);

\draw (3.85,-0.1) -- (-4.15,5);
\draw (3.95,-0.05) -- (-4,5);
\draw (4.05,0) -- (-3.95,5.07);
\draw (4.18,0) -- (-3.95,5.15);

\draw (3.85,0) -- (3.85,5);
\draw (3.95,0) -- (3.95,5);
\draw (4.05,0) -- (4.05,5);
\draw (4.15,0) -- (4.15,5);

\draw (-4.17,0) -- (3.85,5.08);
\draw (-4.05,0) -- (3.95,5.05);
\draw (-3.95,-0.05) -- (4.05,5);
\draw (-3.85,-0.1) -- (4.18,5);

\node at (-7,5) {\Large $V_{i-1}$};
\node at (-7,0) {\Large $V_i$};

\end{tikzpicture}
\end{center}
 \end{minipage}
\begin{minipage}{0.45\textwidth}
  \begin{center}
  \vspace{0.5cm} 
     \begin{tikzpicture}[scale=0.4]

\fill[black] (-4,0) circle (5pt);
\fill[black] (4,0) circle (5pt);
\fill[black] (-4,5) circle (5pt);
\fill[black] (4,5) circle (5pt);

\filldraw[color=black, fill=white, thick](-6,-2) circle (4pt);
\filldraw[color=black, fill=white, thick](-2,-2) circle (4pt);
\filldraw[color=black, fill=white, thick](2,-2) circle (4pt);
\filldraw[color=black, fill=white, thick](6,-2) circle (4pt);

\filldraw[color=black, fill=white, thick](-6,3) circle (4pt);
\filldraw[color=black, fill=white, thick](-2,3) circle (4pt);
\filldraw[color=black, fill=white, thick](2,3) circle (4pt);
\filldraw[color=black, fill=white, thick](6,3) circle (4pt);

\draw[purple, dashed, thick] (-4,0) -- (-6,-2);
\draw[purple, dashed, thick] (-4,0) -- (-2,-2);
\draw[purple, dashed, thick] (4,0) -- (2,-2);
\draw[purple, dashed, thick] (4,0) -- (6,-2);

\draw[purple, dashed, thick] (-4,5) -- (-6,3);
\draw[purple, dashed, thick] (-4,5) -- (-2,3);
\draw[purple, dashed, thick] (4,5) -- (2,3);
\draw[purple, dashed, thick] (4,5) -- (6,3);

\draw (-6.1,2.9) -- (-6.1,-1.9);
\draw (-5.9,2.9) -- (-5.9,-1.9);
\draw (-6.1,2.9) -- (1.9,-1.9);
\draw (-5.9,2.9) -- (2.1,-1.9);
\draw (-6.1,2.9) -- (5.9,-1.9);
\draw (-5.9,2.9) -- (6.1,-1.9);
\draw (-6.1,2.9) -- (-2.1,-1.9);
\draw (-5.9,2.9) -- (-1.9,-1.9);

\draw (-2.1,2.9) -- (-6.1,-1.9);
\draw (-1.9,2.9) -- (-5.9,-1.9);
\draw (-2.1,2.9) -- (-2.1,-1.9);
\draw (-1.9,2.9) -- (-1.9,-1.9);
\draw (-2.1,2.9) -- (1.9,-1.9);
\draw (-1.9,2.9) -- (2.1,-1.9);
\draw (-2.1,2.9) -- (5.9,-1.9);
\draw (-1.9,2.9) -- (6.1,-1.9);

\draw (1.9,2.9) -- (-6.1,-1.9);
\draw (2.1,2.9) -- (-5.9,-1.9);
\draw (1.9,2.9) -- (1.9,-1.9);
\draw (2.1,2.9) -- (2.1,-1.9);
\draw (1.9,2.9) -- (5.9,-1.9);
\draw (2.1,2.9) -- (6.1,-1.9);
\draw (1.9,2.9) -- (-2.1,-1.9);
\draw (2.1,2.9) -- (-1.9,-1.9);

\draw (5.9,2.9) -- (-6.1,-1.9);
\draw (6.1,2.9) -- (-5.9,-1.9);
\draw (5.9,2.9) -- (-2.1,-1.9);
\draw (6.1,2.9) -- (-1.9,-1.9);
\draw (5.9,2.9) -- (1.9,-1.9);
\draw (6.1,2.9) -- (2.1,-1.9);
\draw (5.9,2.9) -- (5.9,-1.9);
\draw (6.1,2.9) -- (6.1,-1.9);

\node at (-7,3) {\Large $V_i'$};
\node at (-7,-2) {\Large $V_{i+1}'$};

\end{tikzpicture}
    \end{center}
    \end{minipage}

\caption{\label{fig:splitting} (Splitting the levels) In the figure, each vertex of each level of $(V,E)$ is split into two new vertices (those connected to it by a dashed edge). Solid edges on the left side of the figure represent $E_i$. Solid edges on the right side of the figure represent $E_{i+1}'$.}
 \end{figure}
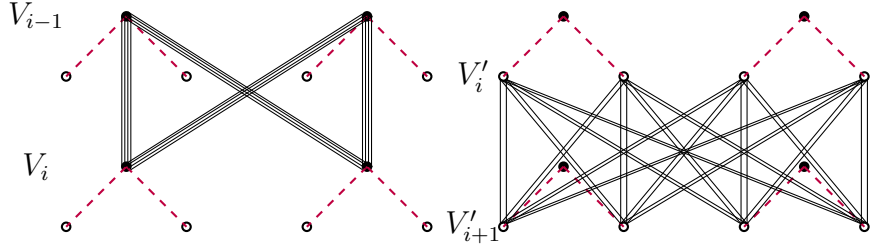

We define the diagram $(V',E')$ as the one given by the matrices $(A_i')_{i\in \N}$ with $A_0'= C_0$ and $A_i'=C_iB_{i-1}$ for $i\geq 1$. Note that $|V_0'|=|V_0|=1$ and $|V_i'|=2|V_{i-1}|=2m_{i-1}$ for all $i\geq 1$.
Note also that the diagrams $(V,E)$ and $(V',E')$ satisfy $K_0(V,E)\cong K_0(V',E')$. Indeed, since  we have the equations $A_i=B_iC_i \quad \forall i\in\N$ and $A_i'=C_iB_{i-1} \quad \forall i\in \N^{\ast}$, $(V,E)$ and $(V',E')$ are the telecopings of $(\Tilde{V},\Tilde{E})$ to the even and odd levels, respectively, where $(\Tilde{V},\Tilde{E})$ is the diagram given by 
$$
\Tilde{V_i} = \begin{cases}   V_i  & \text{ if }    i \text{ is even} \\
V_{\frac{i+1}{2}}' & \text{ if } i \text{ is odd}
\end{cases}
\qquad 
\Tilde{A_i} = \begin{cases}   C_{\frac{i}{2}}  & \text{ if }    i \text{ is even} \\
B_{\frac{i-1}{2}} & \text{ if } i \text{ is odd.}
\end{cases}
$$

 We now construct a sequence $(J_i)_{i\in\N}$ of matrices adapted to $(A_i')_{i\in\N}$ (see Definition \ref{def:adapted}). Using Lemma {\ref{lema:AdaptedMatrices}} we can modify $(V',E')$ to obtain another diagram $(V'',E'')$ with the same dimension group, in which the number of paths between the $i$th level and the root grows faster enough with respect to number of vertices at level $i$. See Figure \ref{fig:factorization} for an illustration of the passage from $(V',E')$ to $(V'',E'')$.

Let $J_0=1$, and for all $i\geq 1$, let $J_i$ be the $2m_{i-1}\times 2m_{i-1}$ integer matrix given by
\[  J_i=S_iD_i\]

where: 
\begin{itemize} 
\item $D_i$ is the $2m_{i-1}\times 2m_{i-1}$ diagonal matrix with positive integer coefficients such that $(D_i)_{k,k}=h_{i-1}$ if $k$ is odd and  $(D_i)_{k,k}=1$ if $k$ is even.
	
\item $S_i$ is a row addition matrix of $2m_{i-1}\times 2m_{i-1} $, such that for a matrix $X$ of $2m_{i-1}\times k$, $S_iX$ is the $2m_{i-1}\times k$ matrix such that
\[ (S_iX)_{2j-1,\bullet}=X_{2j-1,\bullet} \text{ and } (S_iX)_{2j,\bullet}=X_{2j-1,\bullet}+X_{2j,\bullet} \ \text{ for all } 1\leq i\leq m_{i-1}.  \]

\end{itemize} 
For all $i\geq 0$, let $A''_i=J_{i+1}A_i'J_{i}^{-1}$.

{\bf Claim:} We claim that $A_i'J_{i}^{-1}$ is a matrix with entries in $\N^{\ast}$ for all $i\in\N$. The case $i=0$ is trivial. For $i\geq 1$, note  that $A_i'J_{i}^{-1}=C_iB_{i-1}D_{i}^{-1}S_{i}^{-1}$ and that when we multiply $B_{i-1}$ from the right by the matrix $D_{i}^{-1}S_{i}^{-1}$, we first divide the odd columns of $B_{i-1}$ by $h_{i-1}$ and then substract the  $2j$-th column to the divided ($2j-1$)th column of $B_{i-1}$. Therefore, the entries of the odd columns of $B_{i-1}D_{i}^{-1}S_{i}^{-1}$ have the form $q_{i-1}(k,j)-r_{i-1}(k,j)$, and those of the even columns have the form $r_{i-1}(k,j)$, where $q_{i-1}(k,j)$ and $r_{i-1}(k,j)$ are the coefficients of the matrices $Q$ and $R$ used in the splitting of $A_{i-1}$ with respect to $h_{i-1}$. This implies that $B_{i-1}D_{i}^{-1}S_{i}^{-1}$ has integer entries. Let us show that they are positive. This is clear for the even columns. For the odd columns, note that since $h_{i-1}q_{i-1}(k,j)>h_{i-1}^2+h_{i-1}-r_{i-1}(k,j)$ and $r_{i-1}(k,j)< h_{i-1}$ by definition, we have that $q_{i-1}(k,j)>h_{i-1}$ and thus $q_{i-1}(k,j)-r_{i-1}(k,j)>0$. This shows the Claim.

Let $(V'',E'')$ be the Bratteli diagram associated with the matrices $(A_i'')_{i\in \N}$. By the Claim above, the sequence $(J_i)_{i\in\N}$ is adapted to $(A_i')_{i\in\N}$, and by \cref{lema:AdaptedMatrices} we have that $(V'',E'')\sim (V', E')$. This implies that $(G,G^+,u)\cong K_0(V'',E'')$. Note that $|V_i''|=|V_i'|$ for all $i\in \N$, so $|V_0''|=1$ and $|V_i''|=2m_{i-1}$ for all $i\geq 1$.

\vspace{0.5cm}

\begin{figure}[h]

\begin{minipage}{.45\textwidth}
\begin{center}
\begin{tikzpicture}[scale=0.4]

\filldraw[color=black, fill=white, thick](-6,-2) circle (4pt);
\filldraw[color=black, fill=white, thick](-2,-2) circle (4pt);
\filldraw[color=black, fill=white, thick](2,-2) circle (4pt);
\filldraw[color=black, fill=white, thick](6,-2) circle (4pt);

\filldraw[color=black, fill=white, thick](-6,3) circle (4pt);
\filldraw[color=black, fill=white, thick](-2,3) circle (4pt);
\filldraw[color=black, fill=white, thick](2,3) circle (4pt);
\filldraw[color=black, fill=white, thick](6,3) circle (4pt);

\draw (-6.1,2.9) -- (-6.1,-1.9);
\draw (-5.9,2.9) -- (-5.9,-1.9);
\draw (-6.1,2.9) -- (1.9,-1.9);
\draw (-5.9,2.9) -- (2.1,-1.9);
\draw (-6.1,2.9) -- (5.9,-1.9);
\draw (-5.9,2.9) -- (6.1,-1.9);
\draw (-6.1,2.9) -- (-2.1,-1.9);
\draw (-5.9,2.9) -- (-1.9,-1.9);

\draw (-2.1,2.9) -- (-6.1,-1.9);
\draw (-1.9,2.9) -- (-5.9,-1.9);
\draw (-2.1,2.9) -- (-2.1,-1.9);
\draw (-1.9,2.9) -- (-1.9,-1.9);
\draw (-2.1,2.9) -- (1.9,-1.9);
\draw (-1.9,2.9) -- (2.1,-1.9);
\draw (-2.1,2.9) -- (5.9,-1.9);
\draw (-1.9,2.9) -- (6.1,-1.9);

\draw (1.9,2.9) -- (-6.1,-1.9);
\draw (2.1,2.9) -- (-5.9,-1.9);
\draw (1.9,2.9) -- (1.9,-1.9);
\draw (2.1,2.9) -- (2.1,-1.9);
\draw (1.9,2.9) -- (5.9,-1.9);
\draw (2.1,2.9) -- (6.1,-1.9);
\draw (1.9,2.9) -- (-2.1,-1.9);
\draw (2.1,2.9) -- (-1.9,-1.9);

\draw (5.9,2.9) -- (-6.1,-1.9);
\draw (6.1,2.9) -- (-5.9,-1.9);
\draw (5.9,2.9) -- (-2.1,-1.9);
\draw (6.1,2.9) -- (-1.9,-1.9);
\draw (5.9,2.9) -- (1.9,-1.9);
\draw (6.1,2.9) -- (2.1,-1.9);
\draw (5.9,2.9) -- (5.9,-1.9);
\draw (6.1,2.9) -- (6.1,-1.9);

\node at (-7,3) {\Large $V_i'$};
\node at (-7,-2) {\Large $V_{i+1}'$};

\end{tikzpicture}
\end{center}
 \end{minipage}
\begin{minipage}{0.45\textwidth}
  \begin{center}
     \begin{tikzpicture}[scale=0.4]

\filldraw[color=black, fill=white, thick](-6,-2) circle (4pt);
\filldraw[color=black, fill=white, thick](-2,-2) circle (4pt);
\filldraw[color=black, fill=white, thick](2,-2) circle (4pt);
\filldraw[color=black, fill=white, thick](6,-2) circle (4pt);

\filldraw[color=black, fill=white, thick](-6,3) circle (4pt);
\filldraw[color=black, fill=white, thick](-2,3) circle (4pt);
\filldraw[color=black, fill=white, thick](2,3) circle (4pt);
\filldraw[color=black, fill=white, thick](6,3) circle (4pt);

\draw (-6.1,2.9) -- (-6.1,-1.9);
\draw (-5.9,2.9) -- (-5.9,-1.9);
\draw (-6.03,2.85) -- (-6.03,-1.85);
\draw (-5.97,2.85) -- (-5.97,-1.85);

\draw (-6.1,2.9) -- (-2.1,-2.1);
\draw (-5.85,3) -- (-1.9,-1.9);
\draw (-5.97,2.85) -- (-2.1, -1.95);
\draw (-5.89,2.9) -- (-2.05, -1.85);

\draw (-6.1,2.9) -- (1.9,-2.1);
\draw (-5.89,2.85) -- (1.85, -1.97);
\draw (-5.85,2.9) -- (1.87,-1.9);
\draw (-5.82, 2.95) -- (1.97, -1.9);

\draw (-6.1,2.95) -- (5.9,-2.1);
\draw (-5.89, 2.95) -- (5.89, -2.02);
\draw (-5.85, 2.95) -- (5.9, -1.92);
\draw (-5.83,3) -- (6.1,-1.9);

\draw (-2.15,2.95) -- (-6.1,-1.9);
\draw (-2.1,2.9) -- (-5.97,-1.85);
\draw (-2.03, 2.85) -- (-5.9,-1.9);
\draw (-1.9,2.9) -- (-5.85,-1.95);

\draw (-2.1,2.9) -- (-2.1,-1.9);
\draw (-1.9,2.9) -- (-1.9,-1.9);
\draw (-2.03, 2.85) -- (-2.03,-1.85);
\draw (-1.97, 2.85) -- (-1.97, -1.85);

\draw (-2.1,2.9) -- (1.85,-1.95);
\draw (-1.97, 2.85) -- (1.9,-1.9);
\draw (-1.9,2.9) -- (1.99,-1.9);
\draw (-1.85,2.95) -- (2.1,-1.9);

\draw (-2.1,2.9) -- (5.85,-2.05);
\draw (-1.93, 2.9) -- (5.84, -1.95);
\draw (-1.87, 2.93) -- (5.9, -1.9);
\draw (-1.88,3) -- (6.1,-1.9);

\draw (1.85,2.97) -- (-6.1,-1.9);
\draw (1.88,2.9) -- (-5.95,-1.9);
\draw (1.9, 2.85) -- (-5.88, -1.95);
\draw (2.1,2.9) -- (-5.85,-2);

\draw (1.85,2.97) -- (-2.1,-1.9);
\draw (1.88,2.9) -- (-1.97, -1.85);
\draw  (1.97,2.85) -- (-1.9,-1.9);
\draw (2.1,2.9) -- (-1.85,-2);

\draw (1.9,2.9) -- (1.9,-1.9);
\draw (2.1,2.9) -- (2.1,-1.9);
\draw (1.97,2.85) -- (1.97,-1.85);
\draw (2.03, 2.85) -- (2.03, -1.85);

\draw (1.9,2.9) -- (5.85,-2);
\draw (2.05,2.85) -- (5.9,-1.9);
\draw (2.13, 2.89) -- (5.99, -1.89);
\draw (2.15,3) -- (6.1,-1.9);

\draw (5.86,3.08) -- (-6.1,-1.9);
\draw (5.86, 3) -- (-5.9, -1.9);
\draw (5.9, 2.92) -- (-5.85, -1.97);
\draw (6.1,2.9) -- (-5.85,-2.05);

\draw (5.86,3) -- (-2.1,-1.9);
\draw (5.86, 2.92) -- (-1.97, -1.9);
\draw (5.95, 2.9) -- (-1.85,-1.95);
\draw (6.1,2.9) -- (-1.85,-2.05);

\draw (5.86,3) -- (1.9,-1.9);
\draw (5.89, 2.9) -- (2.03, -1.85);
\draw (6,2.88) -- (2.1,-1.93);
\draw (6.1,2.9) -- (2.12,-2.05);

\draw (5.9,2.9) -- (5.9,-1.9);
\draw (6.1,2.9) -- (6.1,-1.9);
\draw (5.97, 2.85) -- (5.97, -1.85);
\draw (6.03, 2.85) -- (6.03, -1.85);

\node at (-7,3) {\Large $V_i''$};
\node at (-7,-2) {\Large $V_{i+1}''$};

\end{tikzpicture}
\end{center}
  \end{minipage}
\caption{ (Factorization) At each level, the number of vertices of $(V'',E'')$ and $(V',E')$ is the same. Edges of $(V'',E'')$ are obtained from the factorization $A_i''=J_{i+1}A_i'J_i^{-1}$.}  \label{fig:factorization}
 \end{figure}
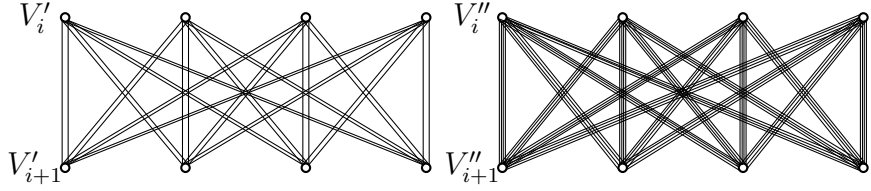

\vspace{0.5cm}
For  the diagram $(V'',E'')$ we  may consider an order as in \cref{lema:injective}. That is, we put any order in $E_1''$, and for all $i\geq 1$, we consider the order given by the morphisms $\tau_i:V_{i+1}''\to V_i''$, defined by
$$\tau_i(v_j)=u_1^{\ell_{1,j}}u_2^{\ell_{2,j}}u_1^{\ell'_{1,j}}u_3^{\ell_{3,j}}u_4^{\ell_{4,j}}\cdots u_{n_i}^{\ell_{n_i,j}} \quad \text{ for all } 1\leq j\leq n_{i+1},$$
where $n_i=|V_i''|$, $n_{i+1}=|V_{i+1}''|$, $\ell_{1,j}=j$, $\ell'_{1,j}=A_{i}''(j,1)-j$ and $\ell_{k,j}=A_{i}''(j,k)$ for all $2\leq k\leq n_i$. Indeed, the condition required to define this order is that the minimum number of edges connecting the first vertex in $V_i''$ with any vertex in $V_{i+1}''$, is strictly larger than $n_{i+1}=2m_{i}$, that is, $A_{i}''(j,1)> 2m_{i}$. In our case, all entries of $A_{i}''$ are greater than $h_{i+1}$ by construction, in particular those of its first column. Since we have chosen $h_{i+1}$ greater than $2m_{i}$, the condition is fulfilled.

Let $\geq$ denote the partial order considered on $E''$. Let $(Y,S)$ be the $\cS$-adic subshift generated by the directive sequence $\bs=(\tau_i)_{i\geq 0}$ of morphisms read on $(V'',E'',\geq)$. Combining \cref{prop:isom} and \cref{lema:injective} we obtain that the Bratteli-Vershik system associated with $(V'',E'', \geq)$ is conjugate to $(Y,S)$. We obtain that $(G,G^+,u)\cong K^0(Y,S)$, which proves the first part of the theorem.

% Remark that by construction we have that $\langle \tau_{[0,i)}\rangle \geq h_{i+1}$ for all $i\geq 1$. Note also that $\textrm{r-comp}(\tau_{i})=|V_{i+2}''|(|V_{i+1}''|+1)$ so we can assume that $\textrm{r-comp}(\tau_{i})\leq |V''_{i+2}|^2=4m_{i+1}^2$.

% For $n\in \N$, $n\geq \|\tau_0\|$, let $i=i(n)$ such that $n\in [\|\tau_{[0,i)}\|, \| \tau_{[0,i+1)}\|)$. Then 
% \cref{prop:Complexitybound} gives us that

% \[p_Y(n) \leq \begin{cases}   3|V_{i+2}''|^3   \cdot n  & \text{ if }  n\in [\|\tau_{[0,i)}\|, \langle  \tau_{[0,i+1)}\rangle)    \\
% 5|V_{i+3}''|^3\cdot n & \text{ if }  n\in [\langle \tau_{[0,i+1)}\rangle , \| \tau_{[0,i+1)}\|)
% \end{cases}  \] 
% For $n\in [\|\tau_{[0,i)}\|, \langle  \tau_{[0,i+1)}\rangle)$, since $n\geq h_{i+1}$, by \eqref{equation:growthh's} we have that

% \[\frac{p_{Y}(n)}{p_n} \leq 24m_{i+1}^3\frac{n}{p_n} \leq 24\frac{1}{i+1}  \]
% while for $ n\in [\langle \tau_{[0,i+1)}\rangle , \| \tau_{[0,i+1)}\|)$, since $n\geq h_{i+2}$, by \eqref{equation:growthh's} 

% \[\frac{p_{Y}(n)}{p_n} \leq 40m_{i+2}^3\frac{n}{p_n}\leq 40\frac{1}{i+2} \]
% From this, we deduce that $p_Y(n)/p_n \to 0$ since $n\to \infty$ implies that $i(n)\to \infty$. This proves the second part of the theorem. 

 Remark that by construction we have that $\langle \tau_{[0,i)}\rangle \geq h_{i}$ for all $i\geq 1$. Note also that $\textrm{r-comp}(\tau_{i})=|V_{i+1}''|(|V_{i}''|+1)$ so we can assume that $\textrm{r-comp}(\tau_{i})\leq |V''_{i}|^2=4m_{i}^2$.

For $n\in \N$, $n\geq \|\tau_0\|$, let $i=i(n)$ such that $n\in [\|\tau_{[0,i)}\|, \| \tau_{[0,i+1)}\|)$. Then 
\cref{prop:Complexitybound} gives us that

\[p_Y(n) \leq \begin{cases}   3|V_{i+1}''|^3   \cdot n  & \text{ if }  n\in [\|\tau_{[0,i)}\|, \langle  \tau_{[0,i+1)}\rangle)    \\
5|V_{i+2}''|^3\cdot n & \text{ if }  n\in [\langle \tau_{[0,i+1)}\rangle , \| \tau_{[0,i+1)}\|)
\end{cases}  \] 
For $n\in [\|\tau_{[0,i)}\|, \langle  \tau_{[0,i+1)}\rangle)$, since $n\geq h_{i}$, by \eqref{equation:growthh's} we have that

\[\frac{p_{Y}(n)}{p_n} \leq 24m_{i}^3\frac{n}{p_n} \leq 24\frac{1}{i}  \]
while for $ n\in [\langle \tau_{[0,i+1)}\rangle , \| \tau_{[0,i+1)}\|)$, since $n\geq h_{i+1}$, by \eqref{equation:growthh's} 

\[\frac{p_{Y}(n)}{p_n} \leq 40m_{i+1}^3\frac{n}{p_n}\leq 40\frac{1}{i+1} \]
From this, we deduce that $p_Y(n)/p_n \to 0$ since $n\to \infty$ implies that $i(n)\to \infty$. This proves the second part of the theorem.

\end{proof}

\begin{proof}[Proof of \cref{theo:main2}]
It suffices to take $(G,G^+,u)$ equal to $K^0(X,T)$ and apply \cref{theo:main1}. The resulting subshift $(Y,S)$ has the desired property on the complexity and is strong orbit equivalent to $(X,T)$ by \cref{theo:caracterizationoe}. 
\end{proof}

For the proof of \cref{thm:divisible_toeplitz} we proceed in a slightly different way, since, as we will see in the proof, having a divisible dimension group gives more flexibility when modifying the Bratteli diagram. Our strategy follows the ideas in the proof of \cite[Theorem 12]{noruegos}, tailored to our situation in order to obtain a Toeplitz subshift $(Y,S)$ having low enough complexity. For the sake of completeness we provide all the details, repeating some arguments of \cite[Theorem 12]{noruegos}.

\begin{proof}[Proof of \cref{thm:divisible_toeplitz}] Suppose that $G$ is a divisible group and let $(V,E)$ be a Bratteli diagram such that $(G,G^+,u)\cong K_0(V,E)$. Let $(A_i)_{i\in \N}$ be the sequence of adjacency matrices of $(V,E)$, where $A_i$ is of $m_{i+1}\times m_i$. We may assume that all the entries of each $A_i$ are positive, that $m_i\geq 2$ for all $i\geq 1$ and that $m_1=2$ with a single edge connecting $v_0$ with each of the two vertices of $V_1$. This is because we can always split the levels where there is only one node in a similar way as in the non-divisible case, to get a new Bratteli diagram with the same dimension group with unit: write the first adjacency matrix as $A_0=B_0C_0$, where $C_0=(1,1)^t$ and $B_0$ is any $|m_1|\times 2$ integer matrix verifying $B_0(j,1)+B_0(j,2)=A_0(j,1)$ for all $1\leq j\leq m_1$; for $i\geq 1$, if $m_i\geq 2$, define $C_i=\Id_{m_i\times m_i}$ and $B_i=A_i$, if $m_i=1$, define $C_i=(1,1)^t$ and $B_i$ as any $|m_{i+1}|\times 2$ integer matrix verifying $B_i(j,1)+B_i(j,2)=A_i(j,1)$ for all $1\leq j\leq m_{i+1}$. Then, define $\Tilde{A_0}=C_0$ and $\Tilde{A_i}=C_iB_{i-1}$ for all $i\geq 1$. The Bratteli diagram associated with the sequence $(\Tilde{A_i})_{i\in\N}$ has the desired property.

Note that since we assumed that $A_0=(1,1)^t$, the first adjacency matrix has the ERS property. The group $G$ corresponds to the inductive limit $\varinjlim_i (\Z^{m_i},A_i)$ and being $G$ divisible, we also have that $\varinjlim_i (\Z^{m_i},A_i)\cong \varinjlim_i (\QQ^{m_i},A_i)$ by \cref{prop:divisible1}. As in the proof of Theorem \ref{theo:main1}, we will modify the diagram $(V,E)$ to obtain a new one with the same dimension group, in which the number of paths between level $i$ and the root of the diagram grows much faster than the number of vertices at $i$. The new diagram $(V',E')$ will 
have the same number of vertices of $(V,E)$ on each level, and the sequence of incidence matrices $(A_i')_{i\in\N}$ will satisfy $A_i'=J_{i+1}A_iJ_i^{-1}$, where the sequence of matrices $(J_i)_{i\in\N}$ is defined inductively. Unlike the proof of Theorem \ref{theo:main1}, and thanks to the divisibility of $G$, here we do not need to require the product matrices $A_iJ_i^{-1}$ to have integer entries, but only that each $J_i$ is positive and invertible over $\QQ$.

For each $i\geq 1$, let $t_i$ be a positive integer such that $\forall t\geq t_i$,

\begin{equation} \label{equation:growtht} \frac{m_{i+1}^3 t}{p_{t}} <\frac{1}{i+1} \end{equation}
We define the sequence of matrices $(A_i')_{i\in\N}$ in the following way. Let $A_0'=A_0$, $J_0=1$, $J_1=\Id_{m_1\times m_1}$ and $k_1=1$. For $1\leq j\leq m_2$, let $\alpha_j$ be the sum of the $j$th row of $A_1$. Let $J_2'$ be the diagonal $m_2\times m_2$ matrix
$$\diag(\alpha_1^{-1},\alpha_2^{-1},\cdots, \alpha_{m_2}^{-1}).$$
Let $s_2\in \N^{\ast}$ be such that $s_2J_2'A_1$ has integer entries (for instance, take $s_2$ to be the least common multiple of the denominators of the entries of $J_2'A_1$), and $\ell_2\in \N^{\ast}$ chosen so that the two following conditions hold,
\begin{itemize}
\item $2s_2\ell_2>t_2$,
\item $2s_2\ell_2\cdot \min_{p,q}\{(J_2'A_1)(p,q)\}>m_2$.
\end{itemize}
Let $k_2=2s_2\ell_2$ and define $J_2=k_2J_2'$, $A_1'=J_2A_1$. Note that $J_2$ is invertible over $\QQ$ and that $A_1'$ is a positive integer matrix whose entries are all divisible by $2$ and strictly greater that $m_2$. Note also that $A_1'$ has the ERS property, with $A_1'(1,1,\cdots, 1)^t=(k_2,k_2,\cdots, k_2)^t$.

Suppose we have defined $J_2,J_3,\cdots, J_{i-1}$. Let $J_i'$ the diagonal invertible matrix over $\QQ$ such that $J_i'A_{i-1}J_{i-1}^{-1}(1,1,\cdots, 1)^t=(1,1,\cdots, 1)^t$. Let $s_i\in \N^{\ast}$ such that  $s_iJ_i'A_{i-1}J_{i-1}^{-1}$ has integer entries, and $\ell_i\in \N^{\ast}$  so that the two following conditions hold,
\begin{itemize}
\item $is_i\ell_i>t_i$,
\item $is_i\ell_i\cdot \min_{p,q}\{(J_i'A_{i-1}J_{i-1}^{-1})(p,q)\}>m_i$.
\end{itemize}
Let $k_i=is_i\ell_i$. Define $J_i=k_iJ_i'$, $A_{i-1}'=J_iA_{i-1}J_{i-1}^{-1}$.
Note that for all $i\geq 0$, the matrix $A_i'$ has the ERS property with $A_i'(1,1,\cdots, 1)^t=(k_{i+1},k_{i+1},\cdots, k_{i+1})^t$. Note also that for all $i\geq 1$, $A_i'$ satisfies the following properties,
\begin{itemize}
    \item $A_i'$ is a positive integer matrix whose entries are all divisible by $i+1$. 
    \item All the entries of $A_i'$ are strictly greater that $m_{i+1}$.
\end{itemize}
Note also that the sequence $(k_i)_{i\geq 1}$ satisfies that $k_i>t_i$ for all $i\geq 2$.
Since all matrices $J_i$ are positive and invertible over $\QQ$, the inductive limits $\varinjlim_i (\QQ^{m_i},A_i)$ and $\varinjlim_i (\QQ^{m_i},A_i')$ are order isomorphic, and the latter is at the same time order isomorphic to $\varinjlim_i (\Z^{m_i},A_i')$, since all the entries of each $A_i'$ are divisible by $i+1$ (see \cref{prop:divisible2}).

Let $(V',E')$ be the Bratteli diagram associated with $\varinjlim_i (\Z^{m_i},A_i')$. By construction, $(G,G^+,u)\cong K_0(V',E')$. Note that for all $i\in\N$, $|V_i'|=|V_i|=m_i$. Since for all $i\geq 1$, each entry of $A_i'$ is strictly greater than $m_{i+1}$, we  may endow $(V',E')$ with the order given in \cref{lema:injective}. As before, let $\geq$ denote the partial order considered on $E'$, and let $(Y,S)$ be the $\cS$-adic subshift generated by the directive sequence $\boldsymbol{\tau}=(\tau_i)_{i\geq 0}$ of morphisms read on $(V',E')$. Combining \cref{prop:isom} and \cref{lema:injective} we obtain that the Bratteli-Vershik system associated with $(V',E', \geq)$ is conjugate to $(Y,S)$. Since the adjacency matrices of $(V',E')$ have the ERS property, \cref{teo:ers} gives us that $(Y,S)$ is a Toeplitz subshift.
We now estimate the complexity $p_Y(n)$.

% Remark that since the matrices $A_i'$ have the ERS property, 
% $$\|\tau_{[0,i)}\|=\langle \tau_{[0,i)}\rangle=k_1k_2\cdots k_{i+1}\geq k_{i+1}.$$ 
% Note also that $\textrm{r-comp}(\tau_{i})=|V_{i+2}'|(|V_{i+1}|'+1)=m_{i+2}(m_{i+1}+1)$, so we can assume that $\textrm{r-comp}(\tau_{i})\leq |V_{i+2}'|^2=m_{i+2}^2$.

% For $n\in \N$, $n\geq \|\tau_0\|$, let $i=i(n)$ such that $n\in [\|\tau_{[0,i)}\|, \| \tau_{[0,i+1)}\|)$. \cref{prop:Complexitybound} gives us that
% \[p_Y(n) \leq    3m_{i+2}^3   \cdot n. \] 
% Since $n\geq \|\tau_{[0,i)}\| \geq  k_{i+1}$, and $k_{i+1}\geq t_{i+1}$, by \eqref{equation:growtht} we have that
% \[\frac{p_{Y}(n)}{p_n} \leq 3m_{i+2}^3\frac{n}{p_n} \leq \frac{3}{i+2}.  \]
% From this, we deduce that $p_Y(n)/p_n \to 0$ since $n\to \infty$ implies that $i(n)\to \infty$. The proof is completed.

Remark that since the matrices $A_i'$ have the ERS property, 
$$\|\tau_{[0,i)}\|=\langle \tau_{[0,i)}\rangle=k_1k_2\cdots k_{i}\geq k_{i}.$$ 
Note also that $\textrm{r-comp}(\tau_{i})=|V_{i+1}'|(|V_{i}|'+1)=m_{i+1}(m_{i}+1)$, so we can assume that $\textrm{r-comp}(\tau_{i})\leq |V_{i+1}'|^2=m_{i+1}^2$.

For $n\in \N$, $n\geq \|\tau_0\|$, let $i=i(n)$ such that $n\in [\|\tau_{[0,i)}\|, \| \tau_{[0,i+1)}\|)$. \cref{prop:Complexitybound} gives us that
\[p_Y(n) \leq    3m_{i+1}^3   \cdot n. \] 
Since $n\geq \|\tau_{[0,i)}\| \geq  k_{i}$, and $k_{i}\geq t_{i}$, by \eqref{equation:growtht} we have that
\[\frac{p_{Y}(n)}{p_n} \leq 3m_{i+1}^3\frac{n}{p_n} \leq \frac{3}{i+1}.  \]
From this, we deduce that $p_Y(n)/p_n \to 0$ since $n\to \infty$ implies that $i(n)\to \infty$. The proof is completed.

\end{proof}

\begin{proof}[Proof of \cref{theo:main3}]
Given any Choquet simplex $K$, we can choose a countable, dense subgroup $G$ of $\Aff(K)$, such that $G$ is a $\QQ$-vector space and contains the constant function $1$. Then, $G$ is divisible and by \cref{theo:simplexanddg} $K$ is affine homeomorphic to $S(G,G^+,1)$, where the positive cone is defined as in \cref{theo:simplexanddg}. Applying \cref{thm:divisible_toeplitz} to $(G,G^+,1)$, we get a Toeplitz subshift $(X,S)$ such that $\lim p_{X}(n)/p_n=0$ and $K^0(X,S)\cong (G,G^+,1)$. By \cite[Theorem 5.5]{HPS}, $\cM(X,S)\cong S(G,G^+,1)$, and then $K\cong \cM(X,S)$.
\end{proof}

We suspect that in \cref{thm:divisible_toeplitz} the condition of $G$ being divisible is superfluous, and that it could be replaced with ``$G$ is the dimension group associated with a Toeplitz subshift", but currently we do not know how to prove it.

It would be interesting to additionally consider ergodic theoretical properties that we can obtain in superlinear complexity subshifts. This has been recently explored in \cite{CKJS}, where the authors constructed examples of subshifts with superlinear complexity admitting loosely Bernoulli and non-loosely Bernoulli ergodic measures.

\end{document}